\documentclass[a4paper,11pt]{amsart}
\usepackage[utf8]{inputenc}
\usepackage{amsfonts}
\usepackage{amsmath}
\usepackage{amssymb}
\usepackage{amsthm}
\usepackage{url}
\usepackage{enumerate}
\usepackage[english]{babel}
\usepackage{paralist}
\usepackage{hyperref}
\usepackage{verbatim}
\usepackage{graphicx}
\usepackage{caption}
\usepackage[]{algorithm2e}

\theoremstyle{plain}
\newtheorem{thm}{Theorem}
\newtheorem{lemma}[thm]{Lemma}
\newtheorem{proposition}[thm]{Proposition}

\theoremstyle{remark}
\newtheorem{remark}{Remark}

\newtheorem{ass}{Assumption}

\newtheorem{example}{Example}

\newcommand{\E}{\mathbb{E}}
\newcommand{\p}{\mathbb{P}}

\newcommand{\R}{\mathbb{R}}
\newcommand{\I}{\mathbb{I}}
\newcommand{\N}{\mathbb{N}}
\newcommand{\F}{\mathcal{F}}
\newcommand{\C}{\mathcal{C}}
\newcommand{\Sc}{\mathcal{S}}

\newcommand{\A}{\mathcal{A}}
\newcommand{\s}{\sigma}
\newcommand{\gr}{\nabla}

\newcommand{\argmin}{\mathop{\rm argmin}\limits}

\DeclareMathOperator{\tr}{Tr}
\newcommand{\he}{\mathrm{H}}

\title{Coupling and a generalised Policy Iteration Algorithm in continuous time}
\date{}

\author{Saul D.\ Jacka}
\address{Department of Statistics, University of Warwick, and\\ The Alan Turing Institute, UK}
\email{s.d.jacka@warwick.ac.uk}

\author{Aleksandar Mijatovi\'c}
\address{Department of Mathematics, King's College London, and\\ The Alan Turing Institute, UK}
\email{aleksandar.mijatovic@kcl.ac.uk}

\author{Dejan \v{S}iraj}
\address{Department of Statistics, University of Warwick, UK}
\email{d.siraj@warwick.ac.uk}

\begin{document}
\sloppy

\thanks{SD supported by the EPSRC grant EP/P00377X/1; AM supported by the EPSRC grant EP/P003818/1
	and the Programme on Data-Centric Engineering  funded by Lloyd's Register Foundation; 
D\v{S} supported by the Slovene Human Resources Development and Scholarship Fund.}



\begin{abstract}
We analyse a version of the policy iteration algorithm for the discounted
infinite-horizon problem for controlled multidimensional diffusion processes,
where both the drift and the diffusion coefficient can be controlled. 
We prove that, under assumptions on the problem data, the payoffs generated by the
algorithm converge monotonically to the value function and an accumulation point 
of the sequence of policies is an optimal policy. The algorithm is stated and 
analysed in continuous time and state, with 
discretisation featuring neither in theorems nor the proofs.  
A key technical tool used to show that the algorithm is well-defined 
is the mirror coupling of Lindvall and Rogers.
\end{abstract}

\date{\today}

\maketitle

%%%%%%%%%%%%%%%%%%%%%%%%%%%%%%%%%%%%%%%%%%%%%%%%%%%%%%%%%%
\section{Introduction}
\label{ch3INT}
%%%%%%%%%%%%%%%%%%%%%%%%%%%%%%%%%%%%%%%%%%%%%%%%%%%%%%%%%%

Howard's policy iteration algorithm (PIA)~\cite{Howard} 
is a well-known  tool for solving control
problems for Markov decision processes (see e.g.~\cite{Bertsekas}
for a recent survey of approximate policy iteration methods for finite
state, discrete time, stochastic dynamic programming problems). 
%The algorithm has proved to be useful in deterministic
%control theory, too (see e.g.\ \cite{Doya} and \cite{Lewis}). Nevertheless, it
The algorithm can be recast for general state-spaces continuous-time control problems,
where allowed actions can be chosen from a Polish space. 
In this paper we investigate the convergence of the PIA
for an infinite horizon discounted cost problem in the context of controlled diffusion processes in $\R^d$, where
the control takes place in an arbitrary compact metric space. 
The main aim of the paper is to analyse the convergence of 
a sequence of policies and the corresponding payoff functions produced by the PIA
under assumptions that are at least in principle verifiable in terms of the model data. 

Our control setting is similar to that of~\cite{Ari}, where an ergodic cost criterion
was considered. The main differences, beyond the cost criterion, are as follows:
(1) we allow the controller to modulate the drift as well as the diffusion coefficient; 
(2) we consider a generalised version of the PIA where an arbitrary scaling function can be used 
to simplify the algorithm;
(3) we investigate the convergence not only of payoffs but also of policies produced by the PIA;
(4) we work with Markov policies that are defined for every $x\in\R^d$, not almost everywhere, and 
obtain a locally uniform convergence of a subsequence to the optimal policy. 

This last point in this list is particularly important in our setting, as our aim is to design and analyse an algorithm 
that can in principle at least be used in to construct an optimal policy. This requirement forces us to work 
in the context of classical solutions of PDEs, rather than relying on generalised solutions of the Poisson equation 
in appropriate Sobolev spaces. The latter approach, followed in~\cite{Ari}, is based on the fact that there exists
a canonical solution to the Poisson equation in $W^{2,p}_{\textrm{loc}}(\R^d)$, see~\cite{Ari_Borkar}.
The analysis of the PIA can than be performed using Krylov’s extension~\cite{Krylov} of It\^o's formula 
to functions  in the Sobolev space $W^{2,p}_{\textrm{loc}}(\R^d)$.

Our method for solving the Poisson equation in the classical sense is based on the coupling of Lindvall and Rogers~\cite{article6}.
This coupling plays a crucial role in the proof of Proposition~\ref{operatorIHM}, guaranteeing that a payoff function for 
a locally Lipschitz Markov policy is the classical solution of the corresponding  Poisson equation. The main technical
contribution of the paper is the result in Lemma~\ref{mdc},
which shows that the mirror coupling from~\cite{article6}
is successful with very high probability, when the diffusion processes are started sufficiently close together. 
Interestingly, the condition in~\cite{article6} for the coupling to be successful is 
not satisfied in our setting in general. 
The proof of Lemma~\ref{mdc} is based on a local path-wise comparison of a time-change of the distance between the coupled diffusions
and an appropriately chosen Bessel process. 

The convergence of the policies and payoffs in the PIA is obtained in several steps. First we show that the 
PIA always improves the payoff. Then we prove, using a ``diagonalisation'' argument and an Arzela-Ascoli type result,
that a subsequence of the policies produced by the PIA converges locally uniformly to a locally Lipschitz limiting 
policy. The final stage of the argument shows that this limiting policy is indeed an optimal policy with payoff 
equal to the pointwise limit of the payoffs produced by the PIA. These steps are detailed in Section~\ref{ch3IHM}
and proved in Section~\ref{sec:Proofs}. 

The literature on the PIA for Markov decision processes in various settings is extensive
(see e.g. \cite{Doshi}, \cite{Lerma}, \cite{Hordi}, \cite{Lasser},
\cite{Meyn}, \cite{Parr}, \cite{Mahadevan}, \cite{Rust}, \cite{Santos} and
\cite{Zhu} and the references therein).
Our approach is in some sense closest to the continuous time analysis in general state spaces 
in~\cite{Doshi}, where the convergence of the subsequence of the policies is established 
in the case of finite action space. 
%Most of the settings of the above papers are discrete, but some are continuous
%(or general) to a certain extent. For example, article \cite{Doshi} deals with
%continuous time Markov decision processes on a fairly general state space, but
%the policy iteration algorithm is only proved to work in the special case of
%finite action space. Paper 
In~\cite{Zhu} this restriction is removed, but the controlled processes considered do not include
diffusions. 
A recent application of the PIA to impulse control in continuous time is given in~\cite{Erhan}.

Finally, as mentioned above, we observe that the PIA can be generalised by multiplying the expression to be minimised by 
an arbitrary positive scaling function that can depend both on the state and the control action (see~\eqref{piaIHM} below).
A choice of the scaling function clearly influences the sequence of policies produced by the algorithm.
In particular, in the one-dimensional case, the scaling function  can be used to eliminate the second derivative of 
the payoff from the algorithm. This idea is described in Section~\ref{ch3IHO}. A numerical example of the PIA 
is reported in Section~\ref{ch3EXA}. In this examples at least, the convergence of the PIA is very fast as the algorithm finds 
an optimal policy in fewer than six iterations.

%%%%%%%%%%%%%%%%%%%%%%%%%%%%%%%%%%%%%%%%%%%%%%%%%%%%%%%%%%
\section{Multidimensional controlled diffusion processes}
\label{ch3IHM}
%%%%%%%%%%%%%%%%%%%%%%%%%%%%%%%%%%%%%%%%%%%%%%%%%%%%%%%%%%

%\subsection{The problem and its algorithmic solution}
%%%%%%%%%%%%%%%%%%%%%%%%%%%%%%%%%%%%%%%%%%%%%%%%%%%%%%%%%%

Let
$(A,d_A)$
be a compact metric space of available control actions and, 
for some $d,m\in\N$,
let 
$\s : \R^d \times A \to \R^{d \times m}$
and 
$\mu : \R^d \times A \to \R^d$
be measurable functions.
Let the set $\A(x)$ 
of \textit{admissible policies} at
$x\in\R^d$
consist of processes 
$\Pi = (\Pi_t)_{t \geq 0}$
satisfying the following:
$\Pi$  is $A$-valued, adapted to a filtration $(\F_t)_{t\geq0}$ and there exists
an $(\F_t)$-adapted process $X^{\Pi,x} = \big( X^{\Pi,x}_t \big)_{t \geq 0}$
satisfying the SDE
\begin{align}
\label{sdeIHM}
X_t^{\Pi,x} = x + \int_0^t \s \left( X_s^{\Pi,x}, \Pi_s \right) \mathrm{d}B_s +
\int_0^t \mu \left( X_s^{\Pi,x},\Pi_s \right) \mathrm{d}s, \quad t \geq 0,
\end{align}
where 
$B = (B_t)_{t \geq 0}$
is an $m$-dimensional $(\F_t)$-Brownian motion. 
Note that the 
filtration $(\F_t)_{t\geq0}$ (and indeed the entire filtered probability space) 
depends on the policy 
$\Pi$ in $\A(x)$.
Pick measurable functions
$\alpha : \R^d \times A \to \R$
and
$f : \R^d \times A \to \R$.
For any
$x \in \R^d$
and
$\Pi \in \A(x)$,
define the \textit{payoff} of the policy $\Pi$ by
\begin{align*}
V_\Pi(x) := \E \left( \int_0^{\infty} \mathrm{e}^{-\int_0^t \alpha \left( X^{\Pi,x}_s, \Pi_s \right) \mathrm{d}s} f \left( X^{\Pi,x}_t, \Pi_t \right) \mathrm{d}t \right).
\end{align*}
\textbf{Control problem.} Construct the \emph{value function}
$V$,
defined by
\begin{equation*}
%\label{value function}
V(x) := \inf_{\Pi \in \A(x)} V_\Pi(x), \quad x \in \R^d,
\end{equation*}
and, if it exists, an optimal control 
$\{\Pi^x\in\A(x):x\in\R^d\}$,
satisfying 
$V(x)=V_{\Pi^x}(x)$.

Note first that 
the problem is specified completely by the deterministic data 
$\s$,
$\mu$,
$\alpha$
and
$f$.
In order to define an algorithm that solves this problem, 
the functions 
$\s$, $\mu$, $\alpha$, $f$ are assumed to satisfy 
Assumption~\ref{ass1IHM}  below
throughout this section. 

\begin{ass}
\label{ass1IHM}
The functions
$\s$,
$\mu$,
$\alpha$
and
$f$
are bounded, and Lipschitz on compacts in
$\R^d \times A$,
i.e. for every compact set
$K \subseteq \R^d\times A$
there exists a constant
$C_K > 0$
such that
\begin{equation}
\label{eq:h_bound_LLiP}
\|h(x,p) - h(y,r)\| \leq C_K \left( \|x-y\|^2 + d_A(p,r)^2 \right)^{\frac{1}{2}}
\end{equation}
holds for every
$(x,p),(y,r) \in K$,
and
$h \in \{ \s, \mu, \alpha, f \}$.
In addition,
$\alpha(x,p)>\epsilon_0>0$ for all $(x,p)\in\R^d\times A$,
and there exists
$\lambda > 0$
such that
\begin{equation}
\label{lambda}
\langle \s(x,p) \s(x,p)^Tv, v\rangle 
%\sum_{i,j=1}^{d} (\s(x,p) \s(x,p)^T)_{i,j} v_i v_j 
\geq \lambda \|v\|^2
\qquad \text{for all} \quad x \in \R^d, p \in A, v \in \R^d.
\end{equation}
\end{ass}

\begin{remark}
\label{rem:matrix_norm}
Here, and throughout the paper,
$\|\cdot\|$
and 
$\langle\cdot,\cdot\rangle$
denote the Euclidean norm and inner product respectively. 
The norm
$\|M\| = \sup \left\{ \|Mv\|/\|v\|: v \in \R^m \backslash \{0\} \right\}=\sqrt{\lambda_{\max}(M M^T)}$,
for a matrix
$M \in \R^{d \times m}$, 
is used in~\eqref{eq:h_bound_LLiP} for $h=\sigma$,
where
%where 
%$\|v\|$
%and 
%$\|Mv\|$
%are the Euclidean norms
%of $v\in\R^m$ and $Mv\in\R^d$, respectively. 
%Denote the transpose of $M$ by
%$M^T\in\R^{m \times d}$ 
%and let 
%Let
$\lambda_{\max}(M M^T)$ is the largest eigenvalue of the non-negative definite matrix $MM^T\in\R^{d\times d}$
and
$M^T\in\R^{m \times d}$ 
denotes the transpose of $M$.
%Note that 
%$\|M\|= \sqrt{\lambda_{\max}(M^T M)}$. 
%which is the norm used in~\eqref{eq:h_bound_LLiP}
%for $h=\s$.
\end{remark}

\begin{remark}
\label{rem:UElip}
The uniform ellipticity condition in~\eqref{lambda} is  the multidimensional analogue of the volatility being bounded away from
$0$.
Hence, 
for all
$x\in\R^d$
and
$p\in A$,
the smallest eigenvalue of
$\s(x,p) \s(x,p)^T\in\R^{d\times d}$
is at least of size 
$\lambda$
and,
in particular, $m\geq d$ (cf. Remark~\ref{rem:d_and_m} below).
\end{remark}

A measurable function
$\pi : \R^d \to A$
is a \emph{Markov policy} (or synonymously \textit{Markov control}) if for 
$x \in \R^d$
there exists an 
$\R^d$-valued
process
$X^{\pi,x} = \left( X^{\pi,x}_t \right)_{t \geq 0}$
that satisfies the following SDE:
\begin{equation}
\label{sde2IHM}
X_t^{\pi,x} = x + \int_0^t \s \left( X_s^{\pi,x}, \pi \left( X_s^{\pi,x} \right) \right) \mathrm{d}B_s +
\int_0^t \mu \left( X_s^{\pi,x}, \pi \left( X_s^{\pi,x} \right) \right) \mathrm{d}s, \quad t \geq 0.
\end{equation}
Let $(\F_t)_{t\geq0}$ be a filtration 
with respect to which
$(X^{\pi,x},B)$ is $(\F_t)$-adapted and $B$ is an  $(\F_t)$-Brownian motion. 
Such a filtration 
$(\F_t)_{t\geq0}$
exists by the definition of a solution of SDE~\eqref{sde2IHM}, see e.g.~\cite[Def.~5.3.1, p.~300]{Karatzas}.
Then
$(\F_t)_{t\geq0}$
can be taken to be the filtration 
in the definition of the policy $\pi(X^{\pi,x})\in\A(x)$. 
Moreover, without loss of generality, we may assume that 
$(\F_t)_{t\geq0}$
satisfies the usual conditions. 

For any function
$\pi : \R^d \to A$
that is Lipschitz on compacts (i.e. locally Lipschitz) in
$\R^d$,
Assumption~\ref{ass1IHM}
implies
(see e.g.~\cite[p.\ 45]{Borodin} and the references therein) 
that the SDE in~\eqref{sde2IHM} has a unique, strong non-exploding solution, 
thus making
$\pi$
a Markov policy.
The payoff function of a locally Lipschitz Markov policy is a classical solution of a linear PDE, a 
fact key for~\eqref{piaIHM} to work. To state this formally,
recall that $h:\R^d\to\R^k$ (for any $k\in\N$) is 
$(1/2)$-H\"older continuous on a compact $D\subset \R^d$
if  $\|h(x')-h(x'')\|\leq K_D \|x'-x''\|^{1/2}$ holds 
for some constant $K_D>0$ and all $x',x''\in D$.
Streamline the notation for Markov policies as follows:
\begin{equation}
\label{eq:Markov_payoff}
\begin{array}{l}
V_\pi(\cdot) := V_{\pi(X^{\pi,\cdot})}(\cdot) \quad \text{and}
\quad L_\pi h := \frac{1}{2} \tr \left( \s_\pi^T \he h \,\s_\pi \right) + \mu_\pi^T \gr h
\quad \text{for} \quad h \in \C^2(\R^d),\\
 \s_\pi(\cdot) := \s(\cdot,\pi(\cdot)), \quad \mu_\pi(\cdot) := \mu(\cdot,\pi(\cdot)),
\quad \alpha_\pi(\cdot) := \alpha(\cdot,\pi(\cdot)), \quad f_\pi(\cdot) := f(\cdot,\pi(\cdot)),
\end{array}
\end{equation}
where
$\he h$ and $\gr h$
are the Hessian and gradient of 
$h$, respectively,  and 
$\tr(M)$ denotes  the trace of any matrix $M\in\R^{m\times m}$.

\begin{proposition}
\label{operatorIHM}
Let Assumption~\ref{ass1IHM} hold.
For a locally Lipschitz Markov policy
$\pi$ we have:
$V_\pi\in\C^2(\R^d)$
is the unique bounded solution of the Poisson equation
$L_\pi V_\pi - \alpha_\pi V_\pi + f_\pi = 0$
and
$\he V_\pi$ is $(1/2)$-H\"older on compacts in $\R^d$.
\end{proposition}

\begin{remark}
\label{rem:d_and_m}
The proof of Proposition~\ref{operatorIHM},
given in Section~\ref{subsec:Proofs_Multidim}, 
depends crucially on the coupling in Lindvall and Rogers~\cite{article6} of two $d$-dimensional diffusions
via a coupling of the corresponding $d$-dimensional driving Brownian motions (see Lemma~\ref{mdc} below). Moreover, it is easy to
see that the probability of coupling could in general be zero if the dimension $m$
of the Brownian noise is strictly less than $d$. In this case the controlled diffusions in $\R^d$, started at distinct points,
could  remain forever on disjoint submanifolds of positive co-dimension
in $\R^d$ for any choice of  control. 
In particular, Proposition~\ref{operatorIHM}
fails in the case $m<d$, as demonstrated by Example~\ref{ex:Prof_fails} below. 
\end{remark}

\begin{example}
\label{ex:Prof_fails}
Let
$m=1$, 
$d=2$, 
$A=[-1,1]$, $f(x_1,x_2,a):=|x_1+x_2|+|x_1-x_2|+a^2/2$, $\alpha(x_1,x_2,a)\equiv 1$, 
$\s(x_1,x_2,a)\equiv(1, 1)^T$ and $\mu(x_1,x_2,a)=a (1 ,-1)^T$.
Then, for a constant policy $\pi_a\equiv a\in A$, the controlled process 
started at $x\in\R^2$ is given by
$X^{\pi,x}_t= x+(1, 1)^T B_t + a (1 ,-1)^T t$, for $t\geq0$.
In particular, for $a=0$ and any $x=(x_1,x_2)^T$, we have
$V_{\pi_0}(x)=|x_1-x_2|+g(x_1+x_2)$, where $g(y):=\E|y+2 B_{e_1}|$, $y\in\R$,  and $e_1$ is an exponential random variable 
with mean $1$, independent of $B$. Since $B_{e_1}$ has a smooth density, it follows that $g$ is also smooth,
implying that $V_{\pi_0}$ cannot satisfy the conclusion of Proposition~\ref{operatorIHM}.
\end{example}

\begin{remark}
\label{rem:d_is_m}
As we are using the standard weak formulation of the control problem (i.e. the filtered probability space is \emph{not} 
specified in advance), all that matters for a Markov policy $\pi$ is the law of the controlled process
$X^{\pi,\cdot}$, which solves the martingale problem in~\eqref{sde2IHM}. 
Moreover, this law is uniquely determined by the symmetric-matrix valued function 
$(x,a)\mapsto \s(x,a)\s(x,a)^T\in\R^{d\times d}$ (and of course the drift $(x,a)\mapsto\mu(x,a)$).
Since the symmetric square root of $\s(x,a)\s(x,a)^T$ in $\R^{d\times d}$ 
satisfies Assumption~\ref{ass1IHM}, in the remainder of the paper we assume, without loss of generality,
that the noise and the controlled process are of the same dimension, i.e. $d=m$ (cf. Remark~\ref{rem:UElip}).
\end{remark}

\begin{remark}
For any locally Lipschitz Markov policy $\pi$,
the process $X^{\pi,\cdot}$ is strong Markov. Hence~\cite[Thm.~1.7]{Dynkin}
implies that $V_\pi$ is in the domain $\mathcal{D}(A_\pi)$ of the generator $A_\pi$ of $X^{\pi,\cdot}$ and that the Poisson equation 
$A_\pi V_\pi - \alpha_\pi V_\pi + f_\pi = 0$ holds.
Recall that for a bounded continuous $g:\R^d\to\R$ 
in $\mathcal{D}(A_\pi)$
the limit 
$(A_\pi g)(x):= \lim_{t\to0} (\E g(X^{\pi,x}_t) - g(x))/t$
exists 
for all $x\in\R^d$.
Furthermore, 
if $g$ is also in $\C^2(\R^d)$,
it is known that $A_\pi g = L_\pi g$.  
However,~\cite[Thm.~1.7]{Dynkin} 
does not imply that $V_\pi$ is in $\C^2(\R^d)$. 
The PDE in Proposition~\ref{operatorIHM}, key for~\eqref{piaIHM} to work,
is established via the coupling argument in Section~\ref{subsec:Proofs_Multidim}. 
\end{remark}

If a policy
$\pi:\R^d\to A$
is constant 
(i.e. $\pi\equiv p\in A$),
write
$\s_p$,
$\mu_p$,
$\alpha_p$,
$f_p$,
$L_p$
and
$V_p$
instead of
$\s_{\pi}$,
$\mu_{\pi}$,
$\alpha_{\pi}$,
$f_{\pi}$,
$L_{\pi}$
and
$V_\pi$,
respectively.
Let 
$S:\R^d\times A\to(0,\infty)$ be a continuous function 
and, for any $p\in A$, denote $S_p(x):=S(x,p)$.
Under Assumption~\ref{ass1IHM}, 
the function
$p \mapsto S_p(x)(L_p h(x)  - \alpha_p(x) h(x) + f_p(x))$, $p \in A$,
is continuous for any $x \in \R^d$ and $h\in \C^2(\R^d)$.
Since $A$ is compact, there exists 
$I_h(x)\in A$, which minimises this function.

\begin{ass}
\label{ass2_and_a_half_IHM}
For any $h\in \C^2(\R^d)$,
the function $I_h:\R^d\to A$ can be chosen to be locally Lipschtiz on $\R^d$. 
The continuous scaling  function  
$S$ 
satisfies 
$\epsilon_S<S<M_S$
for some 
$\epsilon_S,M_S\in (0,\infty)$.
\end{ass}
%We can now state the algorithm.

\begin{algorithm}[H]
\begin{center}
\textbf{Generalised Policy Iteration Algorithm}
\end{center}
\KwIn{$\s$, $\mu$, $\alpha$, $f$, $S$  satisfying Assumptions~\ref{ass1IHM}--\ref{ass2_and_a_half_IHM}, 
constant policy $\pi_0$ and $N\in\N$.}
\For{$n=0$ \KwTo $N-1$}{Compute $V_{\pi_n}$ from the PDE in Proposition~\ref{operatorIHM}.  Choose the policy $\pi_{n+1}$ as follows:
\begin{equation}
\label{piaIHM}
\tag{gPIA}
\pi_{n+1}(x) \in \argmin_{p \in A} \left\{S_p(x) \left( L_p V_{\pi_{n}}(x) - \alpha_p(x) V_{\pi_{n}}(x) + f_p(x) \right)\right\},
\quad x \in \R^d.
\end{equation}
}
\KwOut{Approximation $V_{\pi_N}$ of the value function $V$.}
\end{algorithm}
\smallskip

\begin{remark}
\label{rem:Well_defined_alg}
By Assumption~\ref{ass2_and_a_half_IHM}, the policy $\pi_{n+1}$ defined in~\eqref{piaIHM} is locally Lipschitz 
for all $0\leq n\leq N-1$. 
Hence, by Proposition~\ref{operatorIHM}, the algorithm is well defined. 
\end{remark}

\begin{remark}
In the classical case of~\eqref{piaIHM} 
we take $S\equiv1$.
A non-trivial scaling function $S$, which plays 
an important role in the one-dimensional context (see Section~\ref{ch3IHO} below), 
makes the algorithm  into
a generalised Policy Iteration Algorithm. 
Thm~\ref{decreasingIHM} requires only the positivity and continuity of $S$.
The uniform bounds on $S$ in Assumption~\ref{ass2_and_a_half_IHM}
are used only in the proof of Thm~\ref{verificationIHM}.
\end{remark}

\eqref{piaIHM} always leads to an improved payoff 
(Theorem~\ref{decreasingIHM} is proved in Section~\ref{subsec:Proofs_Multidim} below):
\begin{thm}
\label{decreasingIHM}
Under
Assumptions~\ref{ass1IHM}--\ref{ass2_and_a_half_IHM}, 
the inequality
$V_{\pi_{n+1}} \leq V_{\pi_{n}}$
holds
on
$\R^d$
for all
$n \in \{0,\ldots,N-1\}$.
\end{thm}

The sequence 
$\{ V_{\pi_N} \}_{N \in \N}$, obtained by 
running~\eqref{piaIHM} from a given policy $\pi_0$ for each $N\in\N$, 
is  non-increasing and bounded.  Hence we may define
\begin{equation}
\label{eq:V_lim_def}
V_{\lim}(x) := \lim_{N \to \infty}V_{\pi_N}(x), \quad x \in \R^d.
\end{equation}
However, 
the sequence of the corresponding Markov policies
$\{\pi_N\}_{N\in\N}$
need not converge 
and, 
even if it did, the limit may be discontinuous and hence not necessarily a Markov policy.

\begin{remark}
If the algorithm stops before $N$, i.e.\
$\pi_{n+1} = \pi_n$
for some
$n<N$,
then clearly
$V_{\pi_N} = V_{\pi_n}$
and
$\pi_N = \pi_n$.
As this holds for any $N>n$,
with~\eqref{piaIHM} started at a given $\pi_0$,
we may proceed directly to the verification lemma (Theorem \ref{verificationIHM} below) to conclude that
$V_{\pi_n}$
is the value function with 
an optimal policy 
$\pi_n$.
%In general, we did not succeed to prov that 
%$\pi_N$ 
%converges to an optimal Markov policy. 
%However, we do establish that under our assumptions an optimal Markov policy exists
%and is given as a limit of a subsequence of 
%$\{\pi_N\}_{N\in\N}$.
\end{remark}

Controlling the convergence of the policies requires the following additional assumptions. 
Introduce the set
$\Sc_{B,K}:=\{ h\in \C^2(\R^d): \|\gr h(x)\|<B^1, \|\he h(x)\|<B^2\text{ for }x\in D_K\}$,
where 
$D_K:=\{x\in\R^d:\|x\|\leq K\}$
is  a ball of radius 
$K>0$ and $B:=(B^1,B^2)\in(0,\infty)^2$ are constants.

\begin{ass}
\label{ass2IHM}
For any $K>0$, there exist constants 
$B_K$ and $C_K$ satisfying the following:
if
$h\in\Sc_{B_K,K}$,
then 
$I_h$
(defined in Assumption~\ref{ass2_and_a_half_IHM})
satisfies 
$d(I_h(x),I_h(y))\leq C_K \|x-y\|$ for all 
$x,y\in D_K$.
\end{ass}

\begin{ass}
\label{ass3IHM}
For any 
$K>0$,
let 
$B_K,C_K$ be as in Assumption~\ref{ass2IHM}.
Then for a locally Lipschitz Markov policy 
$\pi:\R^d\to A$, 
such that
$d(\pi(x),\pi(y))\leq C_K \|x-y\|$,
$x,y\in D_K$,
the following holds:
$V_\pi\in\Sc_{B_K,K}$.
\end{ass}

\begin{remark}
Non-trivial problem data that satisfy Assumptions~\ref{ass1IHM}--\ref{ass3IHM}
are described in Section~\ref{ch3EXA} below.
It is precisely these types of examples that motivated the form Assumptions~\ref{ass2_and_a_half_IHM}--\ref{ass3IHM} take.
Assumption~\ref{ass1IHM} is standard and Assumptions~\ref{ass2_and_a_half_IHM}--\ref{ass2IHM} 
concern only the deterministic data specifying the problem. 
Assumption~\ref{ass3IHM} essentially states that $\|\he V_\pi\|$
has a prescribed bound on the ball $D_K$
if the coefficients of the PDE in Proposition~\ref{operatorIHM}
have a prescribed Lipschitz constant. 
Schauder's boundary estimates for elliptic PDEs~\cite[p.~86]{FriedmanParabolic}
suggest that this requirement is both natural and feasible. In fact, Assumption~\ref{ass3IHM}
may follow from assumptions of the type~\ref{ass1IHM}--\ref{ass2IHM} on the problem data.
This is left for future research. 
\end{remark}

Proposition~\ref{subsequenceIHM} and Theorems~\ref{limitsIHM} and~\ref{verificationIHM}, proved in  
Section~\ref{subsec:Proofs_Multidim} below, show that~\eqref{piaIHM} converges.

\begin{proposition}
\label{subsequenceIHM}
Let Assumptions~\ref{ass1IHM}--\ref{ass3IHM} hold. 
Then there exists a subsequence of
$\{ \pi_{N} \}_{N \in \N}$
that converges uniformly on every compact subset of
$\R^d$ to a locally Lipschitz Markov policy. 
\end{proposition}

Let $\pi_{\lim}:\R^d\to A$
denote a locally Lipschitz Markov policy 
that is a locally uniform limit of a subsequence of 
$\{ \pi_{N} \}_{N \in \N}$.
By Propositions~\ref{operatorIHM}, %and~\ref{subsequenceIHM},
$V_{\pi_{\lim}}$ 
is a well-defined function in  $\C^2(\R^d)$
that solves the corresponding Poisson equation. 
However, 
since 
$\pi_{\lim}$
clearly depends on its defining subsequence,  so may
$V_{\pi_{\lim}}$. 
Furthermore,
$V_{\lim}$
may depend on the choice of 
$\pi_0$
in~\eqref{piaIHM}.
%be the payoff for the policy $\pi_{\lim}$.
%and
%$\pi_{\lim}$
%and, 
%additionally,
%$\pi_{\lim}$
But this is not so,
since
$V_{\lim}$
equals both the value function $V$ 
and
the payoff for the policy $\pi_{\lim}$.

\begin{thm}
\label{limitsIHM}
Under Assumptions~\ref{ass1IHM}--\ref{ass3IHM}, 
the equality  $V_{\lim} = V_{\pi_{\lim}}$
holds
on $\R^d$
for a policy
$\pi_{\lim}$. 
\end{thm}

%The final step (Theorem~\ref{verificationIHM} below) implies that 
%$V_{\lim} = V$.

\begin{thm}
\label{verificationIHM}
Let Assumptions~\ref{ass1IHM}--\ref{ass3IHM} hold. 
Then for every
$x \in \R^d$
and
$\Pi \in \A(x)$
the inequality
$V_{\lim}(x) \leq V_{\Pi}(x)$
holds.
Hence 
$V_{\lim}$
equals the value function $V$, does not depend on the choice of
$\pi_0$
in~\eqref{piaIHM}
and
$\pi_{\lim}$
is an optimal locally Lipschitz Markov policy for the control problem. 
\end{thm}

\begin{remark}
The key technical issue in the proof of Theorem~\ref{verificationIHM}
is that the policies in the convergent subsequence constructed in the proof of Proposition~\ref{subsequenceIHM}
are not improvements of their predecessors (cf.~\eqref{piaIHM}). 
The idea of the proof is 
to work with a convergent subsequence of the pairs of policies
$\{(\pi_N,\pi_{N+1})\}_{N\in\N}$, 
where 
$\{\pi_N\}_{N\in\N}$ 
is  produced by the~\eqref{piaIHM}
(see Section~\ref{subsec:Proofs_Multidim} for details).
\end{remark}

%%%%%%%%%%%%%%%%%%%%%%%%%%%%%%%%%%%%%%%%%%%%%%%%%%%%%%%%%%
\section{The one-dimensional case}
\label{ch3IHO}
%%%%%%%%%%%%%%%%%%%%%%%%%%%%%%%%%%%%%%%%%%%%%%%%%%%%%%%%%%

%%%%%%%%%%%%%%%%%%%%%%%%%%%%%%%%%%%%%%%%%%%%%%%%%%%%%%%%%%

There are two reason for considering the one-dimensional control problem in its own right. \\
(A) The canonical choice for the scaling function $S:=1/\s^2$ simplifies the~\eqref{piaIHM}
to 
\begin{equation}
\label{piaIHO}
\pi_{n+1}(x) \in \argmin_{p \in A} 
%\left( \frac{\mu(x,p)}{\s(x,p)^2} \,V_{\pi_{n}}'(x) - \frac{\alpha(x,p)}{\s(x,p)^2} \,V_{\pi_{n}}(x) + \frac{f(x,p)}{\s(x,p)^2} \right)
\left\{ (\mu(x,p) V_{\pi_{n}}'(x) - \alpha(x,p)V_{\pi_{n}}(x) + f(x,p))/\s^2(x,p) \right\},
%\left( \frac{L_p V_{\pi_{n}}}{\s_p^2}(x) - \frac{\alpha_p V_{\pi_{n}}}{\s_p^2}(x) +
%\frac{f_p}{\s_p^2}(x) \right),
%\quad x \in (a,b), \ \ \ n \in \N,
\end{equation}
by removing the second derivative of the payoff function 
$V_{\pi_{n}}$
from the minimisation procedure. 
This reduction appears to make the numerical implementation of
the~\eqref{piaIHM}  converge extremely fast: 
in the example in Section~\ref{subsec:Numerics} below
the optimal payoff and policy are obtained in
fewer than half a dozen iterations. \\
(B) It is natural 
to control the process 
$X^{\Pi,x}$ 
only up to its first exit from an interval 
$(a,b)$, where 
$a,b \in [-\infty,\infty]$,
and generalise the payoff as follows:
%\begin{align*}
%& V_\Pi(x) := \E \left( \int_0^{\tau_a^b \left( X^{\Pi,x} \right)}
%\mathrm{e}^{-\int_0^t \alpha \left( X^{\Pi,x}_s, \Pi_s \right) \mathrm{d}s} f
%\left( X^{\Pi,x}_t, \Pi_t \right) \mathrm{d}t \right. \\
%& \left. \qquad \qquad \quad + \mathrm{e}^{-\int_0^{\tau_a^b \left( X^{\Pi,x}
%\right)} \alpha \left( X^{\Pi,x}_t, \Pi_t \right) \mathrm{d}t}\,
%g \left( X^{\Pi,x}_{\tau_a^b(X^{\Pi,x})} \right)
%\I_{\{ \tau_a^b \left( X^{\Pi,x} \right) < \infty \}} \right).
%\end{align*}
\begin{align*} & 
V_\Pi(x) := \E \left( \int_0^{\tau_a^b( X^{\Pi,x})} \mathrm{e}^{-\int_0^t \alpha_{\Pi_s}(X^{\Pi,x}_s)
\mathrm{d}s} f_{\Pi_s}(X^{\Pi,x}_t) \mathrm{d}t 
  + \mathrm{e}^{-\int_0^{\tau_a^b( X^{\Pi,x})} \alpha_{\Pi_t}(X^{\Pi,x}_t) \mathrm{d}t} g(X^{\Pi,x}_{\tau_a^b(X^{\Pi,x})})
 \right).
\end{align*}
Here
$\mu,\s,\alpha,f : (a,b) \times A \to \R$
are measurable 
with the same notational 
convention as in Section~\ref{ch3IHM} ($f_p(x)=f(x,p)$ \textit{etc.}).
Furthermore, 
$\Pi \in \A(x)$
if
$X^{\Pi,x}$
follows SDE~\eqref{sdeIHM}
on the stochastic interval $[0,\tau_a^b(X^{\Pi,x}))$,
where
$\tau_a^b(X^{\Pi,x}):= \inf \{ t \geq 0;\; X^{\Pi,x}_t \in\{a,b\} \}$ ($\inf \emptyset = \infty$),
and 
$X^{\Pi,x}_{\tau_a^b(X^{\Pi,x})}=X^{\Pi,x}_t$ for $\tau_a^b(X^{\Pi,x})\leq t$
(i.e. $\Pi_t$, $t\in[\tau_a^b(X^{\Pi,x}),\infty)$, are irrelevant for $V_\Pi(x)$).
Pick an arbitrary function 
$g : \{a,b\} \cap \R \to \R$
and set the control problem as in Section~\ref{ch3IHM} with $\R^d$ substituted by $(a,b)$.

\begin{remark}
\label{rem:Improtant}
In Assumptions~\ref{ass1IHM}--\ref{ass3IHM}  we substitute $\R^d$ with $(a,b)$.
In particular, inequality~\eqref{lambda} in Assumption~\ref{ass1IHM} takes the form
$
\s^2(x,p)\geq \lambda
$
for all $x \in(a,b)$, $p \in A$. 
Assumption~\ref{ass1IHM}
hence implies the requirement on the scaling function $S=1/\s^2$ in 
Assumption~\ref{ass2_and_a_half_IHM}. 
In Assumptions~\ref{ass2IHM}--\ref{ass3IHM}, the family of closed balls $(D_K)_{K>0}$
in $\R^d$ is substituted with a family of compact intervals $(D_K)_{K>0}$
in $(a,b)$, such that $\cup_{K>0} D_K = (a,b)$ and, if $K'<K$, 
$D_{K'}$ is contained in the
interior of $D_K$.
\end{remark}

\begin{remark}
On the event 
$\{\tau_a^b(X^{\Pi,x})=\infty\}$
we take
$g(X^{\Pi,x}_{\tau_a^b(X^{\Pi,x})})/\exp(\int_0^{\tau_a^b( X^{\Pi,x})} \alpha_{\Pi_t}(X^{\Pi,x}_t)\mathrm{d}t)=0$,
 since by Assumption~\ref{ass1IHM} the integral is infinite. 
Note also that on this event $X^{\Pi,x}_{\tau_a^b(X^{\Pi,x})}$ may not be defined.
Moreover, if $\{a,b\} \cap \R=\emptyset$, then by Assumption~\ref{ass1IHM} we have
$\tau_a^b(X^{\Pi,x})=\infty$ a.s.
\end{remark}

A measurable function
$\pi : (a,b) \to A$
is a \emph{Markov policy} if for every
$x \in (a,b)$
there exists a 
process
$X^{\pi,x} = \left( X^{\pi,x}_t \right)_{t \geq 0}$
satisfying SDE~\eqref{sde2IHM}
on the stochastic interval 
$[0,\tau_a^b( X^{\pi,x}))$
and 
$X_t^{\pi,x} = X_{\tau_a^b( X^{\pi,x})}^{\pi,x}$  
if $\tau_a^b( X^{\pi,x} ) \leq t < \infty$.
If 
$\pi$
is
a Markov policy,
then
$\pi(X^{\pi,x}) := \left( \pi (  X^{\pi,x}_t ) \right)_{t \geq 0} \in \A(x)$
for every
$x \in (a,b)$,
where we pick arbitrary elements in $A$
for  the values 
$\pi(a)$
and 
$\pi(b)$
if
$a > -\infty$ and $b < \infty$, respectively.
We use analogous notation to that in~\eqref{eq:Markov_payoff}, e.g. 
$L_\pi h := \frac{1}{2}\s_\pi^2 h'' + \mu_\pi h'$ for any $h\in \C^2((a,b))$.

Any locally Lipschitz function 
$\pi : (a,b) \to A$
is a Markov policy, since the Engelbert-Schmidt conditions for the existence and uniqueness of the
solution of the corresponding SDE (see e.g.~\cite[Sec.~5.5]{Karatzas}) are satisfied by
Assumption~\ref{ass1IHM}.
By substituting the state space $\R^d$ with the interval $(a,b)$
in 
Proposition~\ref{operatorIHM}, Theorem~\ref{decreasingIHM}, Proposition~\ref{subsequenceIHM} and  Theorems~\ref{limitsIHM} and~\ref{verificationIHM}
of Section~\ref{ch3IHM}, 
we obtain the results of the present section, which thus solve the control problem in the one-dimensional case. 
In the interest of brevity, we omit their statement. 
We stress that the main difference lies in the fact that the proofs, in Sections~\ref{subsec:aux_one_dim}
and~\ref{subsec:Proofs_OneDim} below, rely on the theory of ODEs and scalar SDEs. 
In particular, we need to prove that the payoff has a continuous extension to a finite boundary point of the
state space.

%The first proposition establishes that there is a large class of Markov policies. 
%\begin{proposition}
%\label{policyIHO}
%If
%$\pi : (a,b) \to A$
%is Lipschitz on compacts in
%$(a,b)$,
%then
%$\pi$
%is a Markov policy.
%\end{proposition}

%It turns out that the corresponding payoff functions satisfy the following differential equation.

%Since
%$\{ V_{\pi_n} \}_{n \in \N}$
%is a decreasing sequence, we define
%$V_{\lim}(x) := \lim_{n \to \infty}V_{\pi_{n}}(x)$, $x \in (a,b)$.
%%The sequence of policies might not converge, but the next proposition says that there exists a convergent subsequence.

%%%%%%%%%%%%%%%%%%%%%%%%%%%%%%%%%%%%%%%%%%%%%%%%%%%%%%%%%%
\section{Examples}
\label{ch3EXA}
%%%%%%%%%%%%%%%%%%%%%%%%%%%%%%%%%%%%%%%%%%%%%%%%%%%%%%%%%%

\subsection{Data satisfying Assumptions~\ref{ass1IHM}--\ref{ass3IHM}} 
\label{subsec:Data_theory}
We now describe a class of models that provably satisfies Assumptions~\ref{ass1IHM}--\ref{ass3IHM}. 
The main aim in the present section is not to be exhaustive, but merely to demonstrate that the 
form (particularly) of Assumptions~\ref{ass2IHM}--\ref{ass3IHM}  
is natural in the context of control problems considered here.
The example we give is in dimension one. But it is clear from the construction below 
that it can easily be generalised.

Let $A:=[-a,a]$, for some constant $a>0$, and $\s,\mu, f, \alpha:\R\times [-a,a]\to\R$
be given by 
\begin{equation}
\label{eq:Example_Class}
\s(x,p):=\s_1(x),\quad \mu(x,p):=\mu_1(x)+p \mu_2,\quad f(x,p):= f_1(x) + f_2(p),\quad
\alpha(x,p)\equiv \alpha_0,
\end{equation}
where $\s_1,\mu_1,f_1\in \C^1(\R)$, $f_2\in\C^2((-a,a))$ is convex and symmetric (i.e. $f_2(p)=f_2(-p)$ for all $p\in A$) 
and $\mu_2$ and $\alpha_0$ are constants.  
For any $h\in\{\s_1,\mu_1,f_1,f_2\}$ (resp. $h'\in\{\s_1',\mu_1',f_1',f_2'\}$) let the positive constant
$C_h$ (resp. $C_h'$) satisfy $|h|\leq C_h$ (resp. $|h'|\leq C_h'$). 
In particular, we assume that the derivatives of 
$\s_1,\mu_1,f_1,f_2$ 
are bounded.
Moreover we may (and do) take 
$C_{f_2}':=f_2'(a)$.
Assume also that $\alpha_0>0$ and $\s_1^2>\lambda>0$ (so that As~\ref{ass1IHM} is satisfied)
and the scaling function $S\equiv1$.  
Then the following proposition holds.

\begin{proposition}
\label{dataIHO}
Assume
$\alpha_0> C_{\mu_1}'+ |\mu_2|(2+C_{f_1}'/C_{f_2}')$ 
and
$|\mu_2|B^2_\infty<L_{f_2}:=\inf_{p\in A} f_2''(p)$, where 
$$
B^2_\infty:=\left[2(C_{f_1}+C_{f_2})+(C_{\mu_1}+ a|\mu_2|)B^1_\infty\right]/\lambda\quad\text{and}\quad
B^1_\infty:=(C_{f_1}'+C_{f_2}')/(\alpha_0- C_{\mu_1}'- |\mu_2|),
$$
then  Assumptions~\ref{ass1IHM}--\ref{ass3IHM} hold.  
Moreover,
in Assumptions~\ref{ass2IHM}--\ref{ass3IHM}  we have
$B_K=(B^1_\infty,B^2_\infty)$ and $C_K=1$
for any 
$K>0$.
\end{proposition}

\begin{remark}
It is clear that the assumptions in Proposition~\ref{dataIHO} define a non-empty subclass of models~\eqref{eq:Example_Class}.
Moreover, these assumptions are much stronger than what is required by our general Assumptions~\ref{ass2IHM}--\ref{ass3IHM} 
since the proposition yields global (rather than local) bounds on the derivatives of the payoff functions 
and the Lipschitz coefficients of the policies 
arising in~\eqref{piaIHM}. 
\end{remark}

\begin{proof}
Pick $h\in\C^2(\R)$, such that 
$|h'(x)|<B_\infty^1$ and $|h''(x)|<B_\infty^2$ 
for all $x\in\R$.
Then the function 
$I_h$ 
in As~\ref{ass2IHM}
satisfies
$I_h(x)=\argmin_{p\in A}\{p\mu_2 h'(x)+f_2(p)\}$.
By assumption we have 
$$
|\mu_2 h'(x)|\leq |\mu_2| B_\infty^1<C'_{f_2}=f_2'(a),\qquad\text{implying}\quad
I_h(x)=(f_2')^{-1}(-\mu_2 h'(x))\quad \forall x\in\R.
$$
Differentiate $I_h$ to obtain 
$|I_h'(x)|\leq |\mu_2|B_\infty^2/L_{f_2}<1$, $x\in\R$,
and note that Assumptions~\ref{ass2_and_a_half_IHM}--\ref{ass2IHM} follow.

We now establish As~\ref{ass3IHM}. The idea is to 
start with any policy $\pi:\R\to A$ in $\C^2(\R)$, such that its derivative satisfies $|\pi'|\leq 1$ on all of $\R$ (e.g. a constant policy),
and apply stochastic flow of diffeomeorphisms~\cite[Sec.~V.10]{Protter} to deduce the necessary regularity 
of the payoff function $V_\pi$.
In the notation from~\eqref{eq:Markov_payoff}, we have
$|\mu_\pi'|\leq C_{\mu_1}'+ |\mu_2|$ 
and
$|\s_\pi'|=|\s_1'|\leq C_{\s_1}'$. 
Hence, for each $x\in\R$, the stochastic exponential $Y=(Y_t)_{t\in\R_+}$, given by
$$
Y_t = 1+\int_0^t \mu_\pi'(X^{\pi,x}_s)Y_s\mathrm{d}s + \int_0^t\s_\pi'(X^{\pi,x}_s)Y_s\mathrm{d}W_s,
$$
exists (we suppress the dependence on $x$ (and $\pi$) from the notation). 
Since the coefficients SDE~\eqref{sde2IHM} are in $\C^1(\R)$ with bounded and locally Lipschitz
first derivative,
\cite[Sec.~V.10, Thm~49]{Protter} implies that
the flow of controlled processes $\{X^{\pi,x}\}_{x\in\R}$  may be constructed on the single 
probability space so that it is smooth in the initial condition $x$ with $\frac{\partial }{\partial x}X^{\pi,x}=Y$. 
The upshot here is that, by the argument  in the proof of~\cite[Prop.~3.2]{Fournie_et_al},
 we obtain a stochastic representation for the derivative 
$\frac{\partial}{\partial x} \E f_\pi(X^{\pi,x}_t)= \E[ Y_t f_\pi'(X^{\pi,x}_t)]$ for every $t\in\R_+$.
%of the function 
%$u_t(x):=\E f_\pi(X^{\pi,x}_t)$, $x\in\R$, given by 
%$u_t'(x) =\E[ Y_t f_\pi'(X^{\pi,x}_t)]$ for every $t\in\R_+$.
Since
$Y_t=M_t \exp \int_0^t\mu_\pi'(X^{\pi,x}_s)\mathrm{d}s$,
where the stochastic exponential $M=\mathcal{E}(\int_0^\cdot\s_\pi'(X^{\pi,x}_s)\mathrm{d}W_s)$
is a true martingale by Novikov's condition,
the following inequality holds:
$\E Y_t \leq \exp( t(C_{\mu_1}'+ |\mu_2|))$.
Since
$\alpha_0>C_{\mu_1}'+ |\mu_2|$ by assumption
and
the inequality $|f_\pi'|\leq C_{f_1}'+ C_{f_2}'$ holds,
we have
$|\E\int_0^\infty \mathrm{e}^{-\alpha_0 s} Y_s f_\pi'(X^{\pi,x}_s) \mathrm{d}s|< B_\infty^1$ 
for all $x\in\R$.

Recall  that
$V_\pi(x)=\E\int_0^\infty \mathrm{e}^{-\alpha_0 s} f_\pi(X^{\pi,x}_s) \mathrm{d}s$.
By~\cite[Sec.~V.8, Thm~43]{Protter},
the family of random variables indexed by $\delta\in(0,1)$,
$$\frac{1}{\delta} \int_0^\infty \mathrm{e}^{-\alpha_0 s} |f_\pi(X^{\pi,x+\delta}_s)-f_\pi(X^{\pi,x}_s)| \mathrm{d}s 
\leq 
\frac{1}{\delta} (C_{f_1}'+ C_{f_2}')\int_0^\infty \mathrm{e}^{-\alpha_0 s} |X^{\pi,x+\delta}_s-X^{\pi,x}_s| \mathrm{d}s,$$
is uniformly integrable. 
Hence 
$\lim_{\delta\to0} (V_\pi(x+\delta)-V_\pi(x))/\delta$
takes the form
$$
V_\pi'(x)=\E\int_0^\infty \mathrm{e}^{-\alpha_0 s} Y_s f_\pi'(X^{\pi,x}_s) \mathrm{d}s,\quad\text{implying }\quad
|V_\pi'(x)|< B_\infty^1 \quad\text{for all $x\in\R$}.
$$
This inequality, the fact
$\s_1^2>\lambda$
and Proposition~\ref{operatorIHM}
imply 
$|V_\pi''|<B_\infty^2$, concluding the proof.
\end{proof}

\begin{remark}
The process $Y$ in the proof of Proposition~\ref{dataIHO}
exists in the multidimensional setting, see~\cite[Sec.~V.10, Thm~49]{Protter}.
Hence the same argument works in higher dimensions if we can deduce a bound
on the Hessian of the payoff function from the PDE in 
Proposition~\ref{operatorIHM}.
\end{remark}

\subsection{Numerical examples}
\label{subsec:Numerics}
%%%%%%%%%%%%%%%%%%%%%%%%%%%%%%%%%%%%%%%%%%%%%%%%%%%%%%%%%%

Consider the one-dimensional control problem:
$A = [-1,1]$,
$a = -10$,
$b = 10$,
$g(a) := a^2$,
$g(b) := b^2$,
$\s(x,p) := 1$,
$\mu(x,p) := p$,
$\alpha(x,p) := 1$
and
$f(x,p) := x^2 + p^2$,
which is in the class discussed in
Section~\ref{subsec:Data_theory}.
%fit into the model presented in Proposition \ref{dataIHO}.
Explicitly, we seek to compute 
$\inf_{\Pi \in \A(x)} V_\Pi(x)$
for every $x \in (-10,10)$, 
where the payoff $V_\Pi(x)$ of  a policy $\Pi$ is defined in Section~\ref{ch3IHO}.

We implemented~\eqref{piaIHM}, with the main step given by~\eqref{piaIHO}, in Matlab. The payoff function at each step is
obtained as the solution to the differential equation from Proposition
\ref{operatorIHM} with the boundary conditions given by the function $g$.
The new policy at each step can be calculated explicitly (cf.\ the proof of Proposition~\ref{dataIHO} above).
Figures~\ref{fig1} and~\ref{fig2} 
graph the payoff functions and the policies (colour coded).
The initial policy 
$\pi_0\equiv 1$  and its payoff correspond to the blue graphs. 

\begin{figure}[!htb]
\vspace{-0.3\baselineskip}
\centering
%\begin{minipage}{.45\textwidth}
\begin{minipage}{.4\textwidth}
  \captionsetup{width=0.81\textwidth}
  \centering
  \includegraphics[trim = 4.3cm 10cm 4.3cm 9.5cm, clip, width=\linewidth]{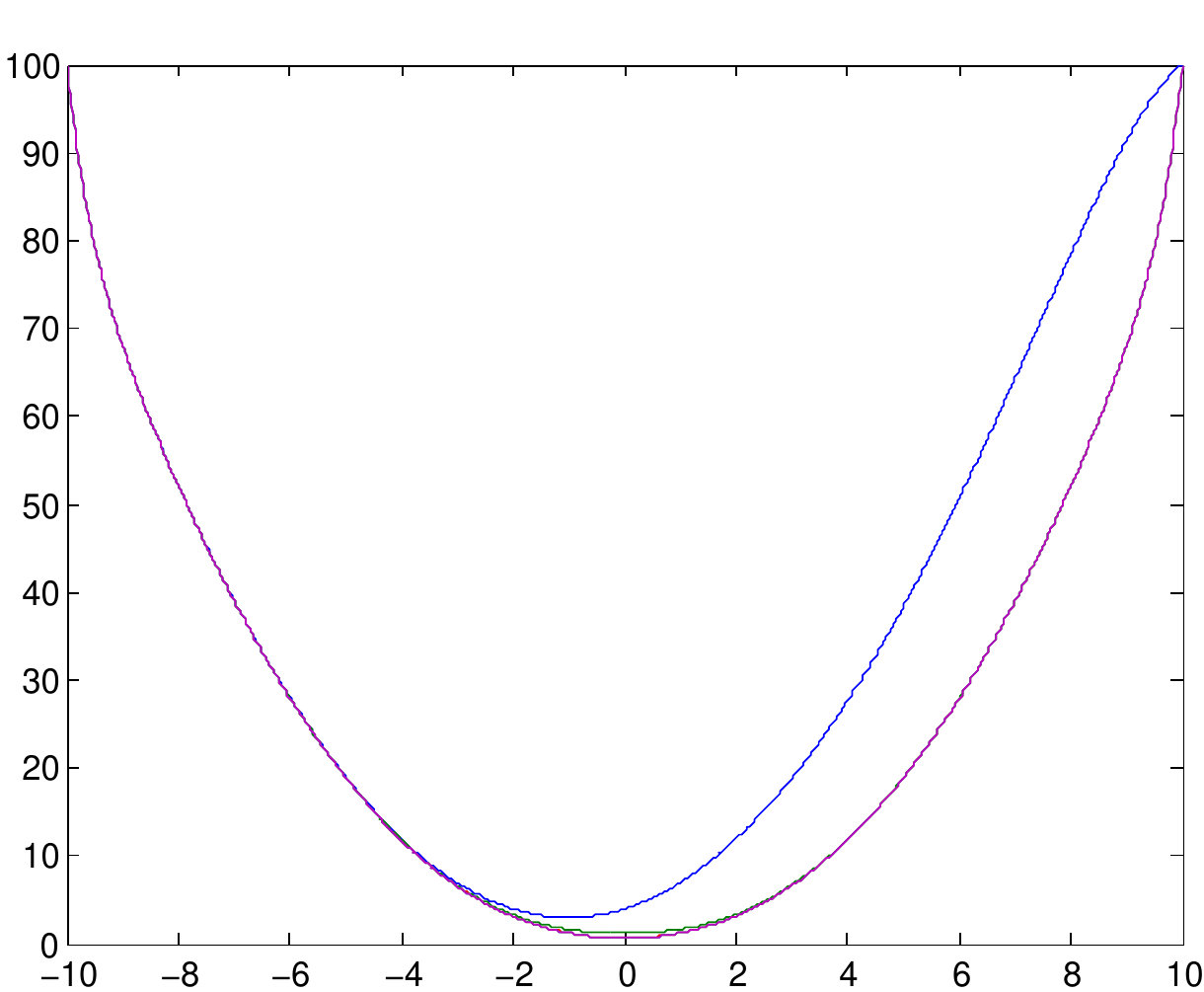}
  \caption{The graphs of $V_{\pi_n}$ for $n \in \{ 0,1,2,3,4 \}$.}
  \label{fig1}
\end{minipage}
\qquad
%\begin{minipage}{.45\textwidth}
\begin{minipage}{.4\textwidth}
  \captionsetup{width=0.81\textwidth}
  \centering
  \includegraphics[trim = 4.3cm 10cm 4.3cm 9.5cm, clip, width=\linewidth]{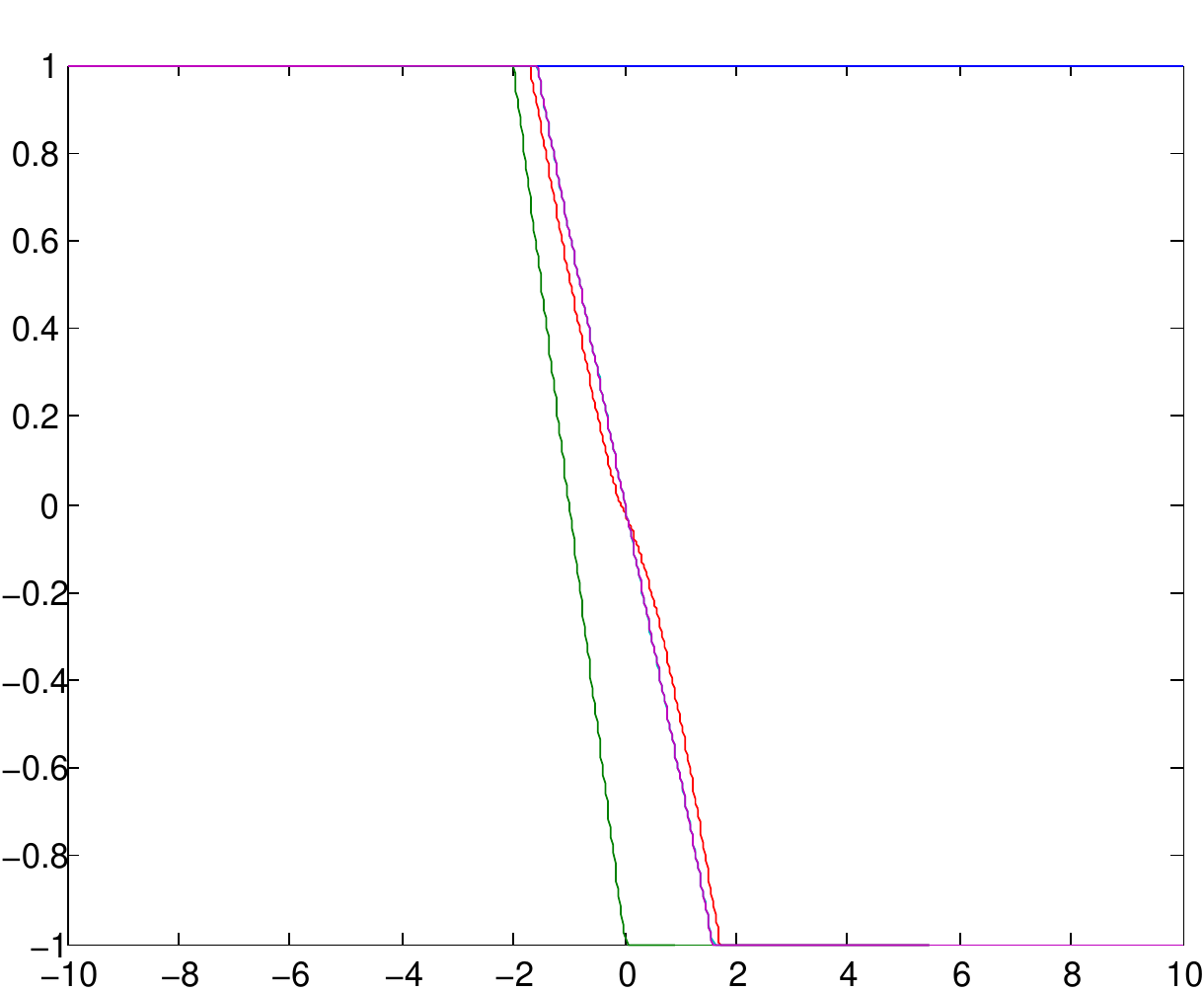}
  \caption{The graphs of $\pi_n$ for $n \in \{ 0,1,2,3,4 \}$.}
  \label{fig2}
\end{minipage}
\vspace{-0.3\baselineskip}
%\caption{A figure with two subfigures}
%\label{fig:test}
\end{figure}

The graphs suggest that convergence effectively occurs in just a few steps.
Figures~\ref{fig3} and~\ref{fig4}, containing the graphs of the differences
of the consecutive payoffs and policies on the logarithmic scale, confirm this. 
In Figures~\ref{fig1} and~\ref{fig2} it seems that fewer graphs are presented 
than is stated in the caption. The reason for this is that the final few graphs
coincide. Moreover, the policies only differ on a subinterval $(-2,2)$, because outside of it 
they coincide as it is optimal to chose one of the boundary points of $A=[-1,1]$.
%are all equal to the same constant (cf.\ the proof of Proposition
%\ref{dataIHO}).  It would be interesting to know if the assumptions can be
%relaxed.
Finally, there is no numerical indication that the sequence of policies have more 
than one accumulation point as they appear to converge very fast indeed.

\begin{figure}[!htb]
\vspace{-0.3\baselineskip}
\centering
\begin{minipage}{.4\textwidth}
  \captionsetup{width=0.61\textwidth}
  \centering
  \includegraphics[trim = 4.3cm 10cm 4.3cm 9.5cm, clip, width=\linewidth]{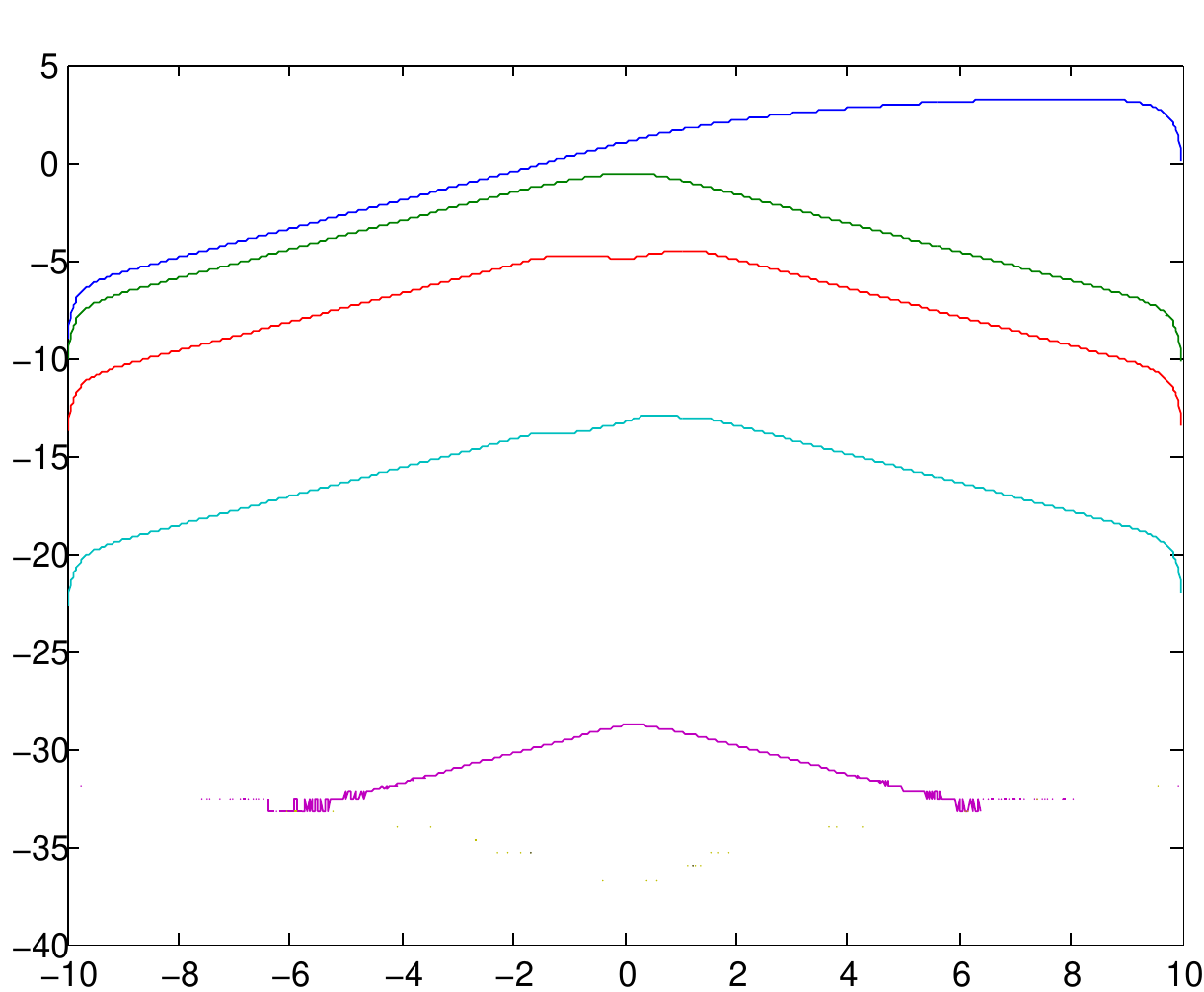}
  %\caption{The graphs of $\log(|V_{\pi_{n+1}}-V_{\pi_n}|)$ for \mbox{$n \in \{ 0,1,2,3,4,5,6,7,99 \}$.}}
  \caption{The graphs of $\log(|V_{\pi_{n+1}}-V_{\pi_n}|)$ for \mbox{$n \in \{ 0,1,2,3,4\}$.}}
  \label{fig3}
\end{minipage}
\qquad
\begin{minipage}{.4\textwidth}
  \captionsetup{width=0.61\textwidth}
  \centering
  \includegraphics[trim = 4.3cm 10cm 4.3cm 9.5cm, clip, width=\linewidth]{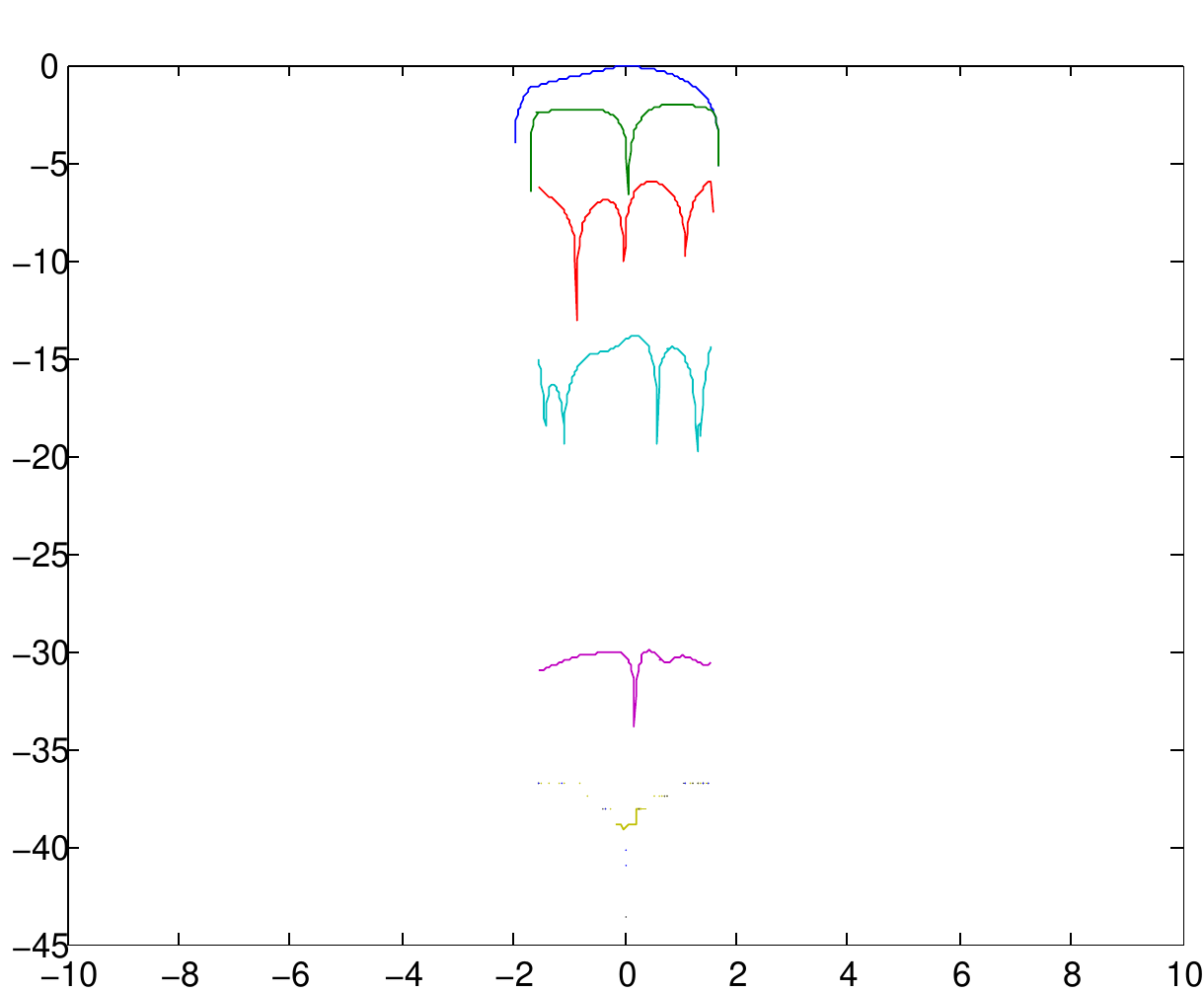}
  %\caption{The graphs of $\log(|\pi_{n+1}-\pi_n|)$ for \mbox{$n \in \{ 1,2,3,4,5,6,7,99 \}$.}}
  \caption{The graphs of $\log(|\pi_{n+1}-\pi_n|)$ for \mbox{$n \in \{0,1,2,3,4\}$.}}
  \label{fig4}
\end{minipage}
\vspace{-0.3\baselineskip}
%\caption{A figure with two subfigures}
%\label{fig:test}
\end{figure}

\section{Proofs}
\label{sec:Proofs}
%%%%%%%%%%%%%%%%%%%%%%%%%%%%%%%%%%%%%%%%%%%%%%%%%%%%%%%%%%

\subsection{Auxiliary results - the multidimensional case}
%%%%%%%%%%%%%%%%%%%%%%%%%%%%%%%%%%%%%%%%%%%%%%%%%%%%%%%%%%

\subsubsection{The reflection coupling of Lindvall and Rogers~\cite{article6} and the continuity of the payoff $V_\pi$}
We now establish the continuity of the payoff function for a locally Lipschitz Markov policy $\pi$ under Assumption~\ref{ass1IHM}.  
The reflection coupling of Lindvall and Rogers~\cite{article6} plays a crucial role in this.
In fact, the continuity of $V_\pi$ hinges on the following property of the coupling in~\cite{article6}: 
copies of $X^{\pi,x}$ and $X^{\pi,x'}$, started very close to each other, will meet with high probability
before moving apart by a certain distance greater than $\|x-x'\|$ (see Lemma~\ref{mdc} below). 

We first show that the coupling from~\cite{article6} can be applied to the diffusion $X^{\pi,\cdot}$.
As explained in Remark~\ref{rem:d_is_m}, we may (and do) assume that the dimension of the noise and the controlled
process are equal, i.e. $d=m$. 
By Assumption~\ref{ass1IHM} above $\s_\pi$ and $\mu_\pi$ are bounded and hence~\cite[As.~(12)(ii)]{article6} holds.
Inequality~\eqref{lambda} in Assumption~\ref{ass1IHM}
implies that 
$\lambda_{\max}(\s_\pi^{-1}\s_\pi^{-T})\leq 1/\lambda$.
Hence, by
Remark~\ref{rem:matrix_norm}, we have 
$\|\s_\pi^{-1}\|\leq 1/\sqrt{\lambda}$
and~\cite[As.~(12)(ii)]{article6} also holds.
The assumptions in~\cite[(12)(i)]{article6} requires 
that 
$\s_\pi$ and $\mu_\pi$ are globally Lipschitz. But this assumption is only used in~\cite{article6} as a 
guarantee that the corresponding SDE has a unique strong solution, which is the case in our
setting under the locally Lipschitz condition in Assumption~\ref{ass1IHM}.
Hence, for any $x,x'\in\R^d$,  the coupling from~\cite{article6} can be applied to construct
the process
$(X^{\pi,x},X^{\pi,x'})$
so that 
$X^{\pi,x}$
follows SDE~\eqref{sde2IHM} and
$X^{\pi,x'}$
satisfies 
$$X^{\pi,x'}_t = x' + \int_0^t \mu_\pi \left( X^{\pi,x'}_s \right) \mathrm{d}s
+ \int_0^t \s_\pi \left( X^{\pi,x'}_s \right) H_s\mathrm{d}B_s,\qquad\text{for $t\in[0,\rho_0(Y))$,}
$$
where
$\rho_0(Y):=\inf\{t\geq 0: \|Y_t\|=0\}$ ($\inf \emptyset := \infty$)
is the coupling time,
$Y:=X^{\pi,x}-X^{\pi,x'}$,
and
\begin{equation}
\label{ref:EQn_H_u_t}
H_t := I - 2u_t u_t^T,\qquad\text{defined via}\quad
u_t := \frac{\s_\pi^{-1} \left( X^{\pi,x'}_t \right) Y_t}{\left\| \s_\pi^{-1} \left( X^{\pi,x'}_t \right) Y_t \right\|}
\qquad\text{for $t\in[0,\rho_0(Y))$,}
\end{equation}
is the reflection in
$\R^d$
about the hyperplane  orthogonal to the unit vector
$u_t$.
Moreover, we have 
$ X_t^{\pi,x'}=X^{\pi,x}_t$ 
for all $t \in[\rho_0(Y),\infty)$. 
Note also that  
$H_t\in O(d)$
is an orthogonal matrix for $t\in[0,\rho_0(Y))$
and the process
$B'=(B'_t)_{t\in\R_+}$,
given by
$B'_t:=\int_0^t (\I_{\left\{ s <\rho_0(Y) \right\}} H_t +\I_{\left\{ s\geq \rho_0(Y) \right\}} I) \mathrm{d}B_s$,
is a Brownian motion by the L\'evy characterisation theorem. 
Hence 
$X^{\pi,x'}$
satisfies the SDE
$\mathrm{d}X^{\pi,x'}_t = \s_\pi(X^{\pi,x'}_t)\mathrm{d} B'_t + \mu_\pi(X^{\pi,x'}_t)\mathrm{d} t$ 
with 
$X^{\pi,x'}_0=x'$
(see~\cite[Sec.~3]{article6} for more details).

\begin{lemma}
\label{mdc}
Fix a locally Lipschitz Markov policy 
$\pi:\R^d\to A$
and 
$x\in\R^d$. 
Then for every 
$\epsilon\in(0,1)$
there exist
$\bar \varphi\in(0,1]$
with the property:
$\forall \varphi\in(0,\bar\varphi)$ 
$\exists \varphi'\in(0,\varphi)$
such that
$\p(\rho_\varphi(Y)<\rho_0(Y))<\epsilon$
if $\|x-x'\|<\varphi'$,
where 
$\rho_c(Y) := \inf \left\{ t \geq 0;\; \|Y_t\| = c \right\}$
($\inf \emptyset = \infty$)
for any $c>0$.
\end{lemma}

\begin{remark}
Note that the main assumption in~\cite[Thm.~1]{article6} is not satisfied in Lemma~\ref{mdc},
as we have no assumption on the global variability of $\s_\pi$.
Hence the coupling 
$(X^{\pi,x},X^{\pi,x'})$
is not necessarily successful  even if the starting points
$x$ and $x'$ are very close to each other,
i.e. possibly 
$\p(\rho_0(Y)<\infty)<1$
even if $\|Y_0\|=\|x-x'\|$ is very close to zero. 
However, by Lemma~\ref{mdc}, the coupling will occur with probability at least $1-\epsilon$
before the diffusions are more than $\bar \varphi=\bar \varphi(\epsilon)$ away from each other, 
implying the continuity of  $V_\pi$ (cf. Lemma~\ref{conIHM} and Remark~\ref{rem:No_coupling}
below).
\end{remark}

\begin{proof}
Let
$\bar{S} := \|Y\|^2$,
$\delta:=\s_\pi ( X^{\pi,x})-\s_\pi (X^{\pi,x'})$
and
$\beta := \mu_\pi ( X^{\pi,x} ) - \mu_\pi ( X^{\pi,x'})$.
Define
\begin{equation}
\label{eq:notation_alpha_beta}
\alpha_t := \s_\pi ( X^{\pi,x}_t)-\s_\pi (X^{\pi,x'}_t)H_t\qquad\text{and}\qquad
v_t := Y_t/\|Y_t\|\qquad\text{for $t\in[0,\rho_0(Y))$.}
\end{equation}
In this proof $x\in\R^d$ is fixed and $x'\in\R^d$ is arbitrary in the ball of radius one centred at $x$.
Recall that
$\gr h(z)=2z$ and $\he h(z) = 2 I$
for $h(z):=\|z\|^2$, $z\in\R^d$,
and apply
It\^{o}'s lemma to $\bar S$: 
\begin{align}
\label{Sbar}
\bar{S}_t & = \|x-x'\|^2 + \int_0^t 2\sqrt{\bar{S}_s} v_s^T \alpha_s  \mathrm{d}B_s 
 + \int_0^t \left( 2\sqrt{\bar{S}_s} v_s^T \beta_s + \tr \left(\alpha_s\alpha_s^T\right) \right) \mathrm{d}s,
\quad\text{$t\in[0,\rho_0(Y))$.}
\end{align}

Our task is to study the behaviour of $\bar S$ when started very close to zero. 
To do this, 
we first establish the facts in~\eqref{eqtrace} and~\eqref{eq:S_def} below, which in turn allow 
us to apply  time-change and  coupling techniques to prove the lemma. 
We start by proving the following:
\begin{equation}
\label{eqtrace}
0\leq \tr \left( \alpha_t \alpha_t^T \right) -  \left\| v_t^T \alpha_t \right\|^2
= \tr \left( \delta_t \delta_t^T \right) - \left\| v_t^T \delta_t \right\|^2 \leq M_x^2 \|Y_t\|^2 \quad\text{for  $t \in [0,\rho_0(Y)\wedge \rho_1(Y))$,}
\end{equation}
where $M_x>1$ is a Lipschitz constant for $\s_\pi$ and $\mu_\pi$ in the ball around $x$ of radius one. %that does not depend on the choice of $x'$ 
The first inequality in~\eqref{eqtrace} follows since the trace is the sum of the eigenvalues of $\alpha_t \alpha_t^T$,
which are all non-negative, while 
$\left\| v_t^T \alpha_t \right\|^2$
is at most the largest eigenvalue. 
The second inequality follows since $\s_\pi$ is  Lipschitz on any ball around $x$ and $\|Y_t\|<1$ for $t< \rho_1(Y)$.
To establish the equality in~\eqref{eqtrace}
note that, as
$\| v_t^T A\|^2=\tr(AA^Tv_tv_t^T)=\tr(v_tv_t^TAA^T)$ for any $A\in \R^{d\times d}$,
we have
$$\tr ( \alpha_t \alpha_t^T ) -  \| v_t^T \alpha_t \|^2 - (\tr ( \delta_t \delta_t^T ) -  \| v_t^T \delta_t \|^2)= 
\tr((I-v_tv_t^T)(\alpha_t \alpha_t^T-\delta_t \delta_t^T)).$$
Recall that $H_t^{-1}=H_t^T=H_t$. We therefore find 
\begin{eqnarray}
\nonumber
\alpha_t \alpha_t^T-\delta_t \delta_t^T & =  &
 \s_\pi(X^{\pi,x'}_t)(I - H_t) \s_\pi( X^{\pi,x}_t)^T + \s_\pi( X^{\pi,x}_t) (I - H_t) \s_\pi (X^{\pi,x'}_t)^T \\
   & = &  2  \left(v_t u_t^T\s_\pi( X^{\pi,x}_t)^T+ \s_\pi( X^{\pi,x}_t)u_t v_t^T\right)\|Y_t\|/ \| \s_\pi (X^{\pi,x'}_t)^{-1} Y_t  \|,
\label{eq:diff_conj}
\end{eqnarray}
where the second equality follows by definition~\eqref{ref:EQn_H_u_t} and identity 
$v_t = Y_t/\|Y_t\|$.
Hence~\eqref{eqtrace} follows. 

Since
$\tr(v_t u_t^T\s_\pi( X^{\pi,x}_t)^T)=\langle\s_\pi( X^{\pi,x}_t)\s_\pi(X^{\pi,x'}_t)^{-1}v_t,v_t\rangle\|Y_t\|/ \| \s_\pi (X^{\pi,x'}_t)^{-1} Y_t  \|$
holds for times
$t \in[0,\rho_0(Y))$,
equalities~\eqref{eqtrace} and~\eqref{eq:diff_conj} yield:
\begin{align*}
\left\| v_t^T \alpha_t \right\|^2 \geq  \left\| v_t^T \alpha_t \right\|^2 - \left\| v_t^T \delta_t \right\|^2 
= 4 \langle\s_\pi( X^{\pi,x}_t)\s_\pi(X^{\pi,x'}_t)^{-1}v_t,v_t\rangle\|Y_t\|^2/ \| \s_\pi (X^{\pi,x'}_t)^{-1} Y_t  \|^2.
\end{align*}
Inequality~\eqref{lambda} in Assumption~\ref{ass1IHM}
implies 
$\|\s_\pi^{-1}\|\leq 1/\sqrt{\lambda}$
(cf. the second paragraph of this section).
Hence 
$\|Y_t\|/ \| \s_\pi (X^{\pi,x'}_t)^{-1} Y_t  \|\geq \sqrt{\lambda}$.
By the definition of $\delta_t$ above we get
\begin{align*}
\left\| v_t^T \alpha_t \right\|^2 
\geq 4 \lambda (1+\langle\delta_t\s_\pi(X^{\pi,x'}_t)^{-1}v_t,v_t\rangle)
\geq 4 \lambda (1-\|\delta_t\| \|\s_\pi(X^{\pi,x'}_t)^{-1}\|)
\geq 4 \lambda (1- \|\delta_t\|/\sqrt{\lambda}).
\end{align*}
For any $\epsilon\in(0,1)$,
%, let 
%$\lambda_0:=4 \lambda (1 - \epsilon)$
define
$$\bar\varphi:=\min\{1,\epsilon\sqrt{\lambda}/M_x, \epsilon(1 - \epsilon)\lambda /M_x^2\},$$ 
where $M_x$ is as in~\eqref{eqtrace} above. 
Then, if $\|x-x'\|<\bar\varphi$ and 
$t\in[0,\rho_{\bar\varphi}(Y))$,
we have
$\|Y_t\|<\bar\varphi$ and hence
$\|\delta_t\|\leq M_x \|Y_t\|\leq \epsilon \sqrt{\lambda}$.
In particular, 
we get
\begin{equation}
\label{eq:S_def}
\| v_t^T \alpha_t\|^2
\geq  4 \lambda (1 - \epsilon)>0 \quad \text{ for any $t\in[0,T)$, where $T:=\rho_0(Y)\wedge \rho_{\bar\varphi}(Y)$.}
\end{equation}

Let 
$M>0$
denote a global upper bound on $\s_\pi$ and 
$\mu_\pi$, 
which exists by Assumption~\ref{ass1IHM}.
Since the inequalities 
$\|v_t^T \alpha_t\|\leq\|\alpha_t\|\leq \|\s_\pi(X^{\pi,x}_t)\| + \|\s_\pi(X^{\pi,x'}_t)\|\leq  2M$ hold for all $t \in[0,\rho_0(Y))$,
the increasing process $[N]=([N]_t)_{t\in\R_+}$,
given by
$[N]_t:=\int_0^t\I_{\left\{s<\rho_0(Y)\right\}} \|v_s^T \alpha_s\|^2  \mathrm{d}s$,
is well-define and  
$[N]_t<\infty$ for every $t\in\R_+$. 
Hence 
$N=(N_t)_{t\in\R_+}$,
given by
$N_t := \int_0^t \I_{\{s<\rho_0(Y)\}}v_s^T \alpha_s  \mathrm{d}B_s$,
is a well-defined local martingale with a quadratic variation process given by 
$[N]$.
Let $\tau=(\tau_s)_{s\in\R_+}$ and $W=(W_s)_{s\in\R^+}$ be 
the Dambis Dubins-Schwartz (DDS) time-change and Brownian motion, respectively, 
for the local martingale $N$ (see~\cite[Thm~3.4.6, p.~174]{Karatzas}). 
More precisely,  let
$s\mapsto \tau_s := \inf \{ t \in\R_+: [N]_t > s \}$
(with $\inf \emptyset =\infty$)
be the inverse of $t\mapsto[N]_t$.
Then
$W$
satisfies
$W_{[N]_t} = N_t$ 
for all 
$t\in\R_+$.
Moreover it holds that  $\tau_s<\infty$ for $s<[N]_\infty:=\lim_{t\uparrow\infty} [N]_t$.
If
$[N]_\infty<\infty$
with positive probability, 
we have to extend the probability space to support $W$ 
(see e.g.~\cite[Prob.~3.4.7, p.~175]{Karatzas}).
This extension however has no bearing on the coupling  
$(X^{\pi,x},X^{\pi,x'})$.

Let 
$\hat \alpha_s:=\alpha_{\tau_{s}}$,
$\hat \delta_s:=\delta_{\tau_{s}}$,
$\hat \beta_s:=\beta_{\tau_{s}}$,
$\hat v_s:=v_{\tau_s}$ 
and
$\hat S_s:=\bar S_{\tau_{s}}$
for $s\in[0,[N]_{\rho_0(Y)})$, cf.~\eqref{eq:notation_alpha_beta} above.
Assume $\|x-x'\|<\bar\varphi$ and
time-change the integrals in~\eqref{Sbar} (see~\cite[Prop.~3.4.8, p.~176]{Karatzas})
to get
\begin{align*}
\hat{S}_u  = \|x-x'\|^2 + \int_0^u 2\sqrt{\hat{S}_s} \,\mathrm{d}W_s
+ \int_0^u \left( 1 + \nu_s\right) \mathrm{d}s,\qquad 
\text{for any $u\in[0,[N]_T)$,}
\end{align*}
where
$\nu_s:= \left(2\sqrt{\hat{S}_s} \hat{v}_s^T \hat{\beta}_s + \tr ( \hat{\delta}_s \hat{\delta}_s^T)
- \| \hat{v}_s^T \hat{\delta}_s \|^2\right)/
\| \hat{v}_s^T \hat{\alpha}_s \|^2$
and
$T$ is defined in~\eqref{eq:S_def}.
By~\eqref{eq:S_def}
it holds that 
$\| \hat{v}_s^T \hat{\alpha}_s \|^2\geq 4 \lambda (1 - \epsilon)$
for all $s\in[0,[N]_T)$.
Then~\eqref{eqtrace} and the definitions of $\hat \beta, \hat \delta$ and $\nu_s$ imply the inequalities 
$0\leq \nu_s< M_x^2\|Y_{\tau_s}\|^2/(\lambda(1-\epsilon))$ for all $s\in[0,[N]_T)$.
Any
$\varphi\in(0,\bar\varphi)$ satisfies 
$\varphi< \epsilon(1-\epsilon)\lambda/M_x^2$
and
$R:=\rho_0(Y)\wedge\rho_\varphi(Y)\leq T$.
Hence the Lipschitz property of $\s_\pi$ and $\nu_\pi$ on the ball of radius $\varphi$ around $x$
implies
\begin{equation}
\label{eq:nu_bound}
\nu_s<\epsilon\qquad
\text{for all $s\in[0,[N]_R)$.}
\end{equation}

The SDE
$S_s = \|x-x'\| + \int_0^s 2\sqrt{S_r}\mathrm{d}W_r + (1 + \epsilon)s$, $s\in\R_+$,
for the squared Bessel process of dimension $1+\epsilon$
has a pathwise unique (and hence strong) solution  
$S=(S_s)_{s\in\R_+}$, see~\cite[App.~A.3, p.~108]{Cerny_Engelbert}. 
Note that $S$ is driven by the DDS Brownian motion $W$ introduced above. 
Hence the coupling $(S,\hat S)$
on the stochastic interval $[0,[N]_R)$
allows us to compare the two processes pathwise.
Assume
$\|x-x'\|<\varphi$. Then  the following equality holds: 
$$\sqrt{S_s} - \sqrt{\hat{S}_s} =\frac{1}{2} \int_0^s \left( \epsilon/\sqrt{S_r}- \nu_r/\sqrt{\hat{S}_r}\right) \mathrm{d}r \qquad \text{for any $s\in[0,[N]_R)$.}$$
Almost surely, the path of the process 
$(\sqrt{S_s} - \sqrt{\hat{S}_s})_{s\in[0,[N]_R)}$
is continuously differentiable and,  by~\eqref{eq:nu_bound}, its derivative is strictly positive at every zero of the path. 
Since the derivative is continuous, it must be strictly positive on a neighbourhood of each zero. This implies that the only 
zero is at $s=0$ (i.e. $S_0=\hat S_0$), 
and it holds that 
\begin{equation}
\label{eq:S_dominates_S_hat}
S_s\geq \hat S_s\qquad\text{for all $s\in[0,[N]_R)$.}
\end{equation}

We now conclude the proof of the lemma. 
Assume as before that 
$\|x-x'\|<\varphi$
and define
$\Upsilon_\varphi(\hat S):=\inf\{s\in[0,[N]_T): \hat S_s=\varphi\}$ (with $\inf \emptyset=\infty$).
Note that the events
$\{\Upsilon_\varphi(\hat S)<\infty\}$
and
$\{[N]_{\rho_\varphi(Y)}<[N]_{\rho_0(Y)}\}$
coincide,
since on either event we have $\Upsilon_\varphi(\hat S)=[N]_R=[N]_{\rho_\varphi(Y)}$.
Hence, 
\begin{equation}
\label{eq:important_inclusion}
\{\rho_\varphi(Y)<\rho_0(Y)\}
= 
\{\Upsilon_\varphi(\hat S)<\infty\} \subseteq \{\text{$S$ exits interval $(0,\varphi)$ at $\varphi$}\},
\end{equation}
where the inclusion follows by~\eqref{eq:S_dominates_S_hat}.
Recall that $s(z)=z^{(1-\epsilon)/2}$, $z\in\R_+$, is a scale function of the diffusion $S$. 
Hence 
$\p(\text{$S$ exits interval $(0,\varphi)$ at $\varphi$})=s(\|x-x'\|)/s(\varphi)$.
Define $\varphi':=\epsilon\varphi$ and note that by~\eqref{eq:important_inclusion} we have:
$\p(\rho_\varphi(Y)<\rho_0(Y))<\epsilon$
for any $x'\in\R^d$ satisfying $\|x-x'\|<\varphi'$.
\end{proof}

\begin{lemma}
\label{conIHM}
Pick a locally Lipschitz Markov policy
$\pi:\R^d\to A$
and let Assumption~\ref{ass1IHM} hold.
Then the corresponding payoff function in~\eqref{eq:Markov_payoff},
$V_\pi:\R^d\to\R_+$,
is continuous.
\end{lemma}

\begin{proof}
Fix $x\in\R^d$ and pick
arbitrary
$\varepsilon \in(0,1)$.
By Assumption~\ref{ass1IHM}
there exists 
$\epsilon_0>0$, such that 
$\alpha_\pi\geq \epsilon_0$,
and a constant
$M>1$
that simultaneously bounds
$\alpha_\pi,|f_\pi|<M$
and
is a Lipschitz constant on the ball of radius one around $x$
for
$\alpha_\pi$
and
$f_\pi$.
Apply Lemma~\ref{mdc} to 
$x,\epsilon:=\varepsilon\epsilon_0/(6M)$ and $\pi$ to obtain
$\bar\varphi\in(0,1]$
such that 
$\forall\varphi\in(0,\bar\varphi)$ $\exists \varphi'\in(0,\varphi)$ such that 
$\p(\rho_\varphi<\rho_0)<\epsilon$
for every $x'\in\R^d$ satisfying $\|x-x'\|<\varphi'$
(here $\rho_\varphi,\rho_0$ stand for 
$\rho_\varphi(Y),\rho_0(Y)$, resp.).
Specifically, define 
\begin{equation}
\label{eq:def_of_phi}
\varphi:=\min\{\bar\varphi/2, \varepsilon /(3(1+M/\epsilon_0)M/\epsilon_0),(\varepsilon\mathrm{e}\epsilon_0)/(3M^2)\}
\end{equation}
and fix
$\varphi'\in(0,\varphi)$ such that the conclusion of Lemma~\ref{mdc} holds.
Throughout this proof we use the notation and notions from 
Lemma~\ref{mdc}. 
In particular, 
$(X^{\pi,x},X^{\pi,x'})$
denotes the coupling of two controlled processes started at $(x,x')$
and we 
assume that
$\|x-x'\|<\varphi'$.

Recall that 
$V_\pi(x')=\E F_\infty(X^{\pi,x'})$ for any
$x'\in\R^d$,
where $F_\infty(X^{\pi,x'})$ is given in~\eqref{eq:F_gains}.
By decomposing the probability space into complementary events 
$\{\rho_\varphi>\rho_0\}$
and
$\{\rho_\varphi<\rho_0\}$,
we obtain the following inequality 
$
|V_\pi(x) - V_\pi(x')|\leq A + A'+ A''$,  
where
\begin{align*}
& A: = \E \left( \I_{\{\rho_\varphi>\rho_0\}} \int_0^{\rho_0} \left|
\mathrm{e}^{-\int_0^t \alpha_\pi \left( X^{\pi,x}_s \right) \mathrm{d}s} f_\pi \left( X^{\pi,x}_t \right)
- \mathrm{e}^{-\int_0^t \alpha_\pi \left( X^{\pi,x'}_s \right) \mathrm{d}s} f_\pi \left( X^{\pi, x'}_t \right)
\right| \mathrm{d}t \right), \\
& A':=\E \left( \I_{\{\rho_\varphi>\rho_0\}} 
\int_{\rho_0}^\infty \left|
\mathrm{e}^{-\int_0^t \alpha_\pi \left( X^{\pi,x}_s \right) \mathrm{d}s} f_\pi \left( X^{\pi,x}_t \right)
- \mathrm{e}^{-\int_0^t \alpha_\pi \left( X^{\pi,x'}_s \right) \mathrm{d}s} f_\pi \left( X^{\pi,x'}_t \right)
\right| \mathrm{d}t  \right), \\
& A'':= \E \left( \I_{\{\rho_\varphi<\rho_0\}} 
\int_0^\infty \left|
\mathrm{e}^{-\int_0^t \alpha_\pi \left( X^{\pi,x}_s \right) \mathrm{d}s} f_\pi \left( X^{\pi,x}_t \right)
- \mathrm{e}^{-\int_0^t \alpha_\pi \left( X^{\pi,x'}_s \right) \mathrm{d}s} f_\pi \left( X^{\pi,x'}_t \right)
\right| \mathrm{d}t \right).
\end{align*}
Hence, by 
Lemma~\ref{mdc}, 
we have
$A''\leq\p(\rho_\varphi<\rho_0) 2M/\epsilon_0<\epsilon2M/\epsilon_0=\varepsilon/3$.

Since in the summands $A$ and $A'$ 
the coupling succeeds before the components of
$(X^{\pi,x},X^{\pi,x'})$
grow at least $\varphi$ apart, 
we can control these terms using the local regularity of $\alpha_\pi$ and $f_\pi$.
Consider $A$. Add and subtract  
$\mathrm{e}^{-\int_0^t \alpha_\pi ( X^{\pi,x}_s ) \mathrm{d}s} f_\pi ( X^{\pi,x'}_t )$
to obtain the bound: 
\begin{align*}
A \leq & 
 \E  \I_{\{\rho_\varphi>\rho_0\}} 
\int_0^{\rho_0} \left(
\mathrm{e}^{-\epsilon_0 t} \left|f_\pi \left( X^{\pi,x}_t \right)
- f_\pi\left( X^{\pi, x'}_t \right)
\right|+M \left|\mathrm{e}^{-\int_0^t \alpha_\pi \left( X^{\pi,x}_s \right)\mathrm{d}s}
- \mathrm{e}^{-\int_0^t \alpha_\pi \left( X^{\pi,x'}_s \right)\mathrm{d}s}\right|\right)\mathrm{d}t.  
\end{align*}
On the event
$\{\rho_\varphi>\rho_0\}$,
for $t<\rho_\varphi$ it holds that 
$\|X^{\pi,x}_t-X^{\pi,x'}_t\|<\varphi$.
Since $z\mapsto \mathrm{e}^{-z}$ has a positive derivative bounded above by one
for $z\in\R_+$,
$\alpha_\pi-\epsilon_0\geq0$
and both $f_\pi$ and $\alpha_\pi$ are Lipschitz with constant $M$ on the ball of radius
$\varphi$ around $x$, we get
\begin{align*}
A \leq & 
 \E  \I_{\{\rho_\varphi>\rho_0\}} 
\int_0^{\rho_0} \left(M\varphi
\mathrm{e}^{-\epsilon_0 t} +M\mathrm{e}^{-\epsilon_0 t}  \int_0^t\left|
\alpha_\pi ( X^{\pi,x}_s ) -  \alpha_\pi ( X^{\pi,x'}_s)\right|\mathrm{d}s\right)\mathrm{d}t
< \varphi\left(\frac{M}{\epsilon_0} +\frac{M^2}{\epsilon_0^2}\right)\leq \frac{\varepsilon}{3},
\end{align*}
where the last inequality follows from~\eqref{eq:def_of_phi}.
Furthermore,
since 
$X^{\pi,x}_t=X^{\pi,x'}_t$ for all 
$t\geq \rho_0$,
it holds that the following expectation 
equals $A'$:
$$\E \I_{\{\rho_\varphi>\rho_0\}} 
\int_{\rho_0}^\infty \mathrm{e}^{-\rho_0\epsilon_0}\left|
\mathrm{e}^{-\int_0^{\rho_0} (\alpha_\pi ( X^{\pi,x}_s )-\epsilon_0) \mathrm{d}s} 
- \mathrm{e}^{-\int_0^{\rho_0} (\alpha_\pi ( X^{\pi,x'}_s )-\epsilon_0) \mathrm{d}s} 
\right| 
\mathrm{e}^{-\int_{\rho_0}^t \alpha_\pi ( X^{\pi,x}_s ) \mathrm{d}s} \left|f_\pi ( X^{\pi,x}_t )\right|
\mathrm{d}t.
$$
Since $|f_\pi|<M$, $\alpha_\pi\geq\epsilon_0$, $|\mathrm{e}^{-z}-\mathrm{e}^{-y}|\leq |z-y|$ for $z,y\in\R_+$
and
$\mathrm{e}^{-t\epsilon_0}\leq\mathrm{e}^{-1}\epsilon_0$
for
$t\in\R_+$
we find
$$
A'\leq 
\E \left( \I_{\{\rho_\varphi>\rho_0\}} M
\mathrm{e}^{-\rho_0\epsilon_0} 
\int_0^{\rho_0} \left|\alpha_\pi ( X^{\pi,x}_s ) -\alpha_\pi ( X^{\pi,x'}_s )\right| \mathrm{d}s 
\right)
\leq M^2\varphi 
\E \left( \I_{\{\rho_\varphi>\rho_0\}} 
\mathrm{e}^{-\rho_0\epsilon_0} \rho_0\right)
\leq \varphi\frac{M^2}{\mathrm{e}\epsilon_0},
$$
which is by~\eqref{eq:def_of_phi} less than $\varepsilon/3$.
Hence, for any 
$\|x-x'\| \leq \varphi'$,
we proved that  
$|V_\pi(x) - V_\pi(x')| < \epsilon$,
which concludes the proof of the lemma. 
\end{proof}

\begin{remark}
\label{rem:No_coupling}
The proofs of Lemmas~\ref{mdc} and~\ref{conIHM} show that if the locally Lipschitz property in 
Assumption~\ref{ass1IHM} is substituted by the globally Lipschitz requirement, we can conclude that
the payoff function $V_\pi$ is in fact uniformly continuous. However, the coupling from~\cite{article6} 
may still not be successful, since  the global Lipschitz condition controls globally
the local variability of the coefficients. The coupling may fail because the assumptions 
in~\cite[Thm.~1]{article6} constrain the global variability of $\s_\pi$. In fact, the idea of the proof of Lemma~\ref{mdc} can be used to
construct an example where 
$\p(\rho_0(Y)<\infty)<1$ by bounding the norm of $\|Y\|^2$ from below by a squared Bessel process
of dimension greater than two on an event of positive probability. 
\end{remark}

\subsubsection{A version of the Ascoli-Arzela Theorem}
The following fact is key for proving the existence of 
the optimal strategy and showing that a subsequence of $\{\pi_N\}_{N\in\N}$
in~\eqref{piaIHM} converges to it.  

\begin{lemma}
\label{AA}
Let
$(M_1,d_1)$
and
$(M_2,d_2)$
be compact metric spaces, and for every
$n \in \N$
let
$f_n : M_1 \to M_2$.
If the sequence
$\{ f_n \}_{n \in \N}$
is equicontinuous, i.e.
$$\forall \epsilon > 0 \quad \exists \delta > 0 \quad \forall x,y \in M_1 \quad \forall n \in \N : \quad
d_1(x,y) < \delta \implies d_2(f_n(x),f_n(y)) < \epsilon,$$
then there exists a uniformly convergent subsequence $\{f_{n_k}\}_{k\in\N}$, i.e. $\exists f:M_1\to M_2$ such that 
for every $\epsilon>0$ there exists $N\in\N$ such that $\sup_{x\in M_1}d_2(f_{n_k}(x),f(x))<\epsilon$ for all $k\geq N$.
\end{lemma}

\begin{proof}
Let $B( x, 1/m) := \{ y \in M_1: d_1(x,y) < 1/m \}$
be a ball of radius $1/m$, $m\in\N$,  centred at $x\in M_1$.
Since $M_1$ is compact and metric, it is totally bounded:  
$\exists S_m\subseteq M_1$ finite satisfying 
$M_1=\cup_{x\in S_m}B(x,1/m)$.
Then $S:=\cup_{m\in\N} S_m=\{ x_n \in M_1;\ n \in \N \}$ is countable and dense in $M_1$.
We now apply the standard diagonalisation argument to find the subsequence in
the lemma.

Let $\iota_1:\N\to\N$ be
an increasing function defining a  subsequence
$\{ f_{\iota_1(n)} \}_{n \in \N}$ that converges at
$x_1$, i.e.
$\lim_{n\to\infty}f_{\iota_1(n)}(x_1)$ exists in $M_2$.
Such a function $\iota_1$ exists since $M_2$ is compact.
Assume now that we have constructed an increasing
$\iota_k:\N\to\N$ such that 
$\{ f_{\iota_k(n)} \}_{n \in \N}$
converges on the set 
$\{ x_1,\ldots,x_k\}$ for some $k\in\N$.
Then there exists an increasing
$\iota:\N\to\N$ 
such that the sequence of functions 
$\{ f_{\iota_{k+1}(n)} \}_{n \in \N}$,
where 
$\iota_{k+1}:=\iota_k\circ\iota$,
converges at 
$x_{k+1}$
as well as on the set 
$\{ x_1,\ldots,x_k\}$,
as it is a subsequence of 
$\{ f_{\iota_k(n)} \}_{n \in \N}$.
Since $k\in \N$ was arbitrary, we have defined a sequence of subsequences of
$\{ f_n \}_{n \in \N}$, such that the $k$-th subsequence converges on 
$\{ x_1,\ldots,x_k\}$.

Consider the ``diagonal'' subsequence  
$\{ f_{n_k} \}_{k \in \N}$,
$f_{n_k}:= f_{\iota_k(k)}$ for any $k \in \N$.
By construction it converges on $S$.
We now prove that it is uniformly Cauchy, 
which implies uniform convergence since
$M_2$ is complete. 
Pick any $\epsilon>0$.
By equicontinuity 
$\exists m \in \N$
such that for any 
$k\in\N$
and
$x,y\in M_1$ satisfying
$d_1(x,y) < 1/m$,
it holds that
$d_2(f_{n_k}(x),f_{n_k}(y)) < \epsilon/3$.
Furthermore, since $S_m$ is finite, 
$\exists N \in \N$
such that
for all natural numbers $k_1,k_2\geq N$
we have
$d_2(f_{n_{k_1}}(y),f_{n_{k_2}}(y)) < \epsilon/3$
for all $y\in S_m$.
Finally, for any 
$x \in M_1$
there exists 
$y\in S_m$
such that $d_1(x,y)<1/m$.
Hence,
for any $k_1,k_2\geq N$
it holds that 
$$ d_2(f_{n_{k_1}}(x),f_{n_{k_2}}(x)) \leq 
d_2(f_{n_{k_1}}(x),f_{n_{k_1}}(y)) +
d_2(f_{n_{k_1}}(y),f_{n_{k_2}}(y)) +
d_2(f_{n_{k_2}}(y),f_{n_{k_2}}(x)) 
< \epsilon.$$
Since $x\in M_1$ was arbitrary, the lemma follows. 
\end{proof}

\subsubsection{A uniformly integrable martingale}   
If the process $X^{\pi,x}$ in~\eqref{sde2IHM}, controlled by a Markov policy $\pi$, exists for all $x\in\R^d$,
then 
$X^{\pi,\cdot}$ 
is a strong Markov process~\cite[Thm~4.30, p.\ 322]{Karatzas}, 
since $\s$ and $\mu$ are bounded by Assumption~\ref{ass1IHM}.
Define the additive functional 
$F(X^{\pi,x})=(F_t(X^{\pi,x}))_{t\in[0,\infty]}$, 
\begin{equation}
\label{eq:F_gains}
F_t(X^{\pi,x}):=\int_0^t\mathrm{e}^{-\int_0^u \alpha_\pi \left( X^{\pi,x}_s \right)\mathrm{d}s}
f_\pi \left( X^{\pi,x}_u \right) \mathrm{d}u \qquad \text{for $t\in[0,\infty]$. }
\end{equation}
\begin{remark}
\label{rem:V_F_bounded}
Note that 
$V_\pi(x)=\E F_\infty(X^{\pi,x})$ and,
by Assumption~\ref{ass1IHM}, the process $|F(X^{\pi,x})|$ is bounded by some constant $C_0>0$. Hence $|F_\infty(X^{\pi,x})|<C_0$
and $|V_\pi(x)|<C_0$.
\end{remark}

\begin{lemma}
\label{martingaleIHM}
The following holds for every Markov policy
$\pi$,
$x \in \R^d$
and $(\F_t)$-stopping time
$T$:
\begin{align*}
 \E \big( F_\infty(X^{\pi,x})\vert \F_T \big) = F_T(X^{\pi,x}) + \I_{\{ T < \infty\}} \mathrm{e}^{-\int_0^{T} \alpha_\pi \left( X^{\pi,x}_s \right) \mathrm{d}s}\,
V_\pi\left( X^{\pi,x}_{T} \right).
\end{align*}
In particular, the process
$M=(M_r)_{r\in[0,\infty]}$
is a uniformly integrable martingale, where
\begin{align*}
M_r := F_r(X^{\pi,x}) + \I_{\{ r < \infty\}} \mathrm{e}^{-\int_0^{r} \alpha_\pi \left( X^{\pi,x}_s \right) \mathrm{d}s} \,V_\pi\left( X^{\pi,x}_{r} \right). 
\end{align*}
\end{lemma}

\begin{proof}
The following calculations imply the lemma: 
\begin{align*}
& \E \big( F_\infty(X^{\pi,x})\vert \F_T \big) = F_T(X^{\pi,x})
+ \E \left( \I_{\{ T < \infty \}} \int_0^{\infty}
\mathrm{e}^{-\int_0^{t + T} \alpha_\pi \left( X^{\pi,x}_s \right) \mathrm{d}s}
f_\pi \left( X^{\pi,x}_{t + T} \right) \mathrm{d}t \,\middle\vert\, \F_{T} \right) \\
& = F_T(X^{\pi,x}) + \I_{\{ T < \infty \}} \mathrm{e}^{-\int_0^{T} \alpha_\pi \left( X^{\pi,x}_t \right) \mathrm{d}t}
\,\E \left( \int_0^{\infty} \mathrm{e}^{-\int_0^{t} \alpha_\pi \left( X^{\pi,x}_{s + T} \right) \mathrm{d}s}
f_\pi \left( X^{\pi,x}_{t + T} \right) \mathrm{d}t \,\middle\vert\, \F_{T} \right) \\
& = F_T(X^{\pi,x}) + \I_{\{ T < \infty \}} \mathrm{e}^{-\int_0^{T} \alpha_\pi \left( X^{\pi,x}_t \right) \mathrm{d}t}
\,V_\pi(X^{\pi,x}_{T}),
\end{align*}
where we applied the strong Markov property of $X^{\pi,\cdot}$ in the last step.
\end{proof}

\subsection{Proofs of results in Section~\ref{ch3IHM}}
\label{subsec:Proofs_Multidim}
\begin{proof}[Proof of Proposition \ref{operatorIHM}]
Assume $m=d$, cf. Remark~\ref{rem:d_is_m}.
It suffices to prove that the PDE holds on the ball 
$D:=\{y\in\R^d:\|y-x'\|<1\}$
for any 
$x'\in\R^d$.
Fix
$x \in D$
and define
$\tau:=\inf\{t\in\R_+:X^{\pi,x}_t\in\partial D\}$ (with $\inf\emptyset =\infty$)
to be the first time the process
$X^{\pi,x}$
hits the boundary 
$\partial D:=\{y\in\R^d:\|y-x'\|=1\}$
of
$D$.
Note that, by  Assumption~\ref{ass1IHM}, we have
$\tau<\infty$.

Let
$v \in \C^2(D) \cap \C(\bar{D})$, where  $\bar D:=D\cup \partial D$,
denote 
a solution of the boundary value problem
$$L_\pi v - \alpha_\pi v + f_\pi = 0\quad \text{in} \quad D, \qquad \text{where $v = V_\pi$ on  $\partial D$.}$$
Since $\pi$ is locally Lipschitz
and~\eqref{eq:h_bound_LLiP}
in Assumption~\ref{ass1IHM} holds, the coefficients 
$\s_\pi,\mu_\pi,f_\pi,\alpha_\pi$
are
$(1/2)$-H\"{o}lder (in fact Lipschitz) on $\bar D$. 
The boundary data $V_\pi|_{\partial D}$ is continuous by Lemma~\ref{conIHM},
$\alpha_\pi \geq 0$
and
$\s_\pi$ satisfies~\eqref{lambda}. 
Hence, by~\cite[Thm~19, p.\ 87]{FriedmanParabolic}, the function
$v$ exists, is unique and  $\he v$ is $(1/2)$-H\"older.

Note that, 
for all $t\in[0,\infty]$,
we have
$X^{\pi,x}_{t \wedge \tau}\in \bar D$.
Hence
we can define
\begin{equation}
\label{eq:def_of_proc_S}
Y_t :=  F_{t \wedge \tau}(X^{\pi,x}) +
\mathrm{e}^{-\int_0^{t \wedge \tau} \alpha_{\pi} \left( X^{\pi,x}_r \right) \mathrm{d}r}
v \left( X^{\pi,x}_{t \wedge \tau} \right),\qquad  \text{for $t\in[0,\infty]$,}
\end{equation}
where the process
$F_\cdot(X^{\pi,x})$
is given in~\eqref{eq:F_gains} above. 
The process 
$Y = (Y_t)_{t \in[0,\infty]}$
is bounded by a constant
and by definition converges almost sure 
$\lim_{t\to \infty}Y_t= Y_\infty$.
Since $v$ solves the boundary value problem above
and
$X^{\pi,x}$
satisfies SDE~\eqref{sde2IHM},
It\^o's formula on the stochastic interval
$[0,\tau]\subset\R_+$
yields
\begin{align*}
Y_t 
& = v(x) + \int_0^{t \wedge \tau} \mathrm{e}^{-\int_0^s \alpha_{\pi} \left( X^{\pi,x}_r \right) \mathrm{d}r}
 \gr v\left( X^{\pi,x}_s \right)^T \s_\pi \left( X^{\pi,x}_s \right) \mathrm{d}B_s, \quad t \in[0,\infty],
\end{align*}
making
$Y$
into a local martingale. 
Since $Y$ is bounded, it is a uniformly integrable martingale
satisfying $v(x)=Y_0=\E[ Y_\infty]$. 
Since
$v=V_\pi$ on $\partial D$
and
$X^{\pi,x}_\tau\in \partial D$,
the definition of $Y$ in~\eqref{eq:def_of_proc_S}
and
Lemma~\ref{martingaleIHM} (applied to the stopping time $T:=\tau$)
yield
\begin{align*}
Y_\infty 
& =  F_\tau(X^{\pi,x}) +
\mathrm{e}^{-\int_0^{\tau} \alpha_{\pi} \left( X^{\pi,x}_r \right) \mathrm{d}r}
V_\pi \left( X^{\pi,x}_{\tau} \right) = \E \left( F_\infty(X^{\pi,x}) \middle\vert \F_{\tau} \right),
\end{align*}
implying 
$v(x)=\E ( F_\infty(X^{\pi,x}))=V_\pi(x)$.

The uniqueness follows similarly: let $v$ be another bounded solution of the Poisson equation on $\R^d$.
Define the process $Y$ as in~\eqref{eq:def_of_proc_S} with $\tau\equiv\infty$ and $t<\infty$.
As above we have $v(x)=\E Y_t$ for all $t\in\R_+$. Then the DCT, applicable since $v$ is bounded,
yields $v(x)=\lim_{t\uparrow\infty} \E Y_t = V_\pi(x)$.
\end{proof}

\begin{proof}[Proof of Theorem \ref{decreasingIHM}]
Let
$\pi_n$ and $\pi_{n+1}$
be as in~\eqref{piaIHM}, $n\in\N\cup\{0\}$.
Define 
$Y=(Y_t)_{t\in\R_+}$ by
\begin{align}
\label{eq:S_defin}
& Y_t :=  F_t( X^{\pi_{n+1},x}) + \mathrm{e}^{-\int_0^t \alpha_{\pi_{n+1}} \left( X^{\pi_{n+1},x}_r \right) \mathrm{d}r}
\,V_{\pi_n} \left( X^{\pi_{n+1},x}_{t} \right), \quad t \in \R_+.
\end{align}
where 
$F_\cdot(X^{\pi_{n+1},x})$
is given in~\eqref{eq:F_gains} above. 
Define 
$\tau_m:=\inf\{t\geq0:\|X^{\pi_{n+1},x}_t\|=m\}$
for any fixed $m>\|x\|$
and note that $\tau_m<\infty$
by  Assumption~\ref{ass1IHM}.
It\^o's formula, applicable by Proposition~\ref{operatorIHM}, yields
\begin{equation*}
Y_{\cdot \wedge \tau_m} = V_{\pi_n}(x)
+ M + \int_0^{\cdot \wedge \tau_m} \mathrm{e}^{-\int_0^s \alpha_{\pi_{n+1}} \left( X^{\pi_{n+1},x}_r \right) \mathrm{d}r}
\left( f_{\pi_{n+1}} + L_{\pi_{n+1}} V_{\pi_n} - \alpha_{\pi_{n+1}} V_{\pi_n} \right)
\left( X^{\pi_{n+1},x}_s \right) \mathrm{d}s,
\end{equation*}
where $M=(M_t)_{t\in\R_+}$,
$M_t:= \int_0^{t \wedge \tau_m} \mathrm{e}^{-\int_0^s \alpha_{\pi_{n+1}}
( X^{\pi_{n+1},x}_r) \mathrm{d}r}\,
( \gr V_{\pi_n}^T \s_{\pi_{n+1}}) ( X^{\pi_{n+1},x}_s ) \mathrm{d}B_s$,
is a local martingale. 
Since the functions
$\s_{\pi_{n+1}}$
and
$\gr V_{\pi_n}$
are bounded on the ball
$\{y\in\R^d:\|y\|\leq m\}$
by Assumption~\ref{ass1IHM} and Proposition~\ref{operatorIHM}, respectively,
and $\alpha_{\pi_{n+1}}>\epsilon_0>0$, the quadratic variation of 
$M$ 
is bounded above by a constant. Hence $M$ is a uniformly integrable martingale. 
In particular, $\E M_t=0$ for all $t\in\R_+$. 
By~\eqref{piaIHM} 
and
Proposition~\ref{operatorIHM}, 
we have 
\begin{equation*} 
\left(
f_{\pi_{n+1}} + L_{\pi_{n+1}} V_{\pi_n} - \alpha_{\pi_{n+1}} V_{\pi_n}\right)
S_{\pi_{n+1}}
\> \leq\> 
\left(
f_{\pi_n} + L_{\pi_n} V_{\pi_n} - \alpha_{\pi_n} V_{\pi_n}\right) S_{\pi_n}= 0 
\qquad \text{on $\R^d$.}
\end{equation*}
Since
$S_{\pi_{n+1}}>0$,
we have
$E \left( Y_{t \wedge \tau_m} \right) \leq V_{\pi_n}(x)$.
Hence~\eqref{eq:S_defin}, Assumption~\ref{ass1IHM} and the DCT, 
as  
$t\uparrow\infty$,
yield
\begin{eqnarray*}
V_{\pi_n}(x) & \geq 
\E F_{\tau_m}( X^{\pi_{n+1},x})  + 
\E 
V_{\pi_n} \left( X^{\pi_{n+1},x}_{\tau_m} \right)
\mathrm{e}^{-\int_0^{\tau_m}
\alpha_{\pi_{n+1}} \left( X^{\pi_{n+1},x}_r \right) \mathrm{d}r}.
%\label{eq:final_inequality_for_monotonicity}
\end{eqnarray*}
Hence, 
by 
Remark~\ref{rem:V_F_bounded},
we have
$V_{\pi_n}(x)\geq \E F_{\tau_m}( X^{\pi_{n+1},x})  - 
C_0 \E \mathrm{e}^{-\epsilon_0\tau_m}$.
Since 
$X^{\pi_{n+1},x}$
satisfies SDE~\eqref{sde2IHM} for all $t\in\R_+$,
%the path $t\mapsto X^{\pi_{n+1},x}_t$
%is continuous and hence
we have
$\lim_{m\uparrow\infty}\tau_m=\infty$. 
The DCT and
Remark~\ref{rem:V_F_bounded}
yield
$V_{\pi_{n+1}}(x)=  \E F_\infty( X^{\pi_{n+1},x})=\lim_{m\uparrow\infty}\E F_{\tau_m}( X^{\pi_{n+1},x})- C_0 \E \mathrm{e}^{-\epsilon_0\tau_m}\leq V_{\pi_n}(x)$,
which concludes the proof.
\end{proof}

\begin{proof}[Proof of Proposition \ref{subsequenceIHM}]
Run~\eqref{piaIHM} to produce a sequence of policies 
$\{\pi_N\}_{N\in\N}$, starting from a constant policy $\pi_0$.
Fix an arbitrary $K_0>0$ and consider the restriction of this sequence 
onto the closed ball $D_{K_0}$.
Since the Lipschitz constant of $\pi_0$ is equal to zero and hence smaller than $C_{K_0}$,
Assumption~\ref{ass3IHM}
implies 
$V_{\pi_0}\in\Sc_{B_{K_0},K_0}$.
Assumption~\ref{ass2IHM}
implies that the Lipschitz constant 
of $\pi_1$ is also at most $C_{K_0}$.
Iterating this argument implies that all the policies 
in the sequence
$\{\pi_N\}_{N\in\N}$
have the same Lipschitz constant on $D_{K_0}$,
making it equicontinuous on $D_{K_0}$. 
By Lemma~\ref{AA} above, there exists a subsequence
that converges uniformly on $D_{K_0}$ to a function $\pi_\infty^0:D_{K_0}\to A$.
Moreover, $\pi_\infty^0$ is also Lipschitz with a constant bounded above by $C_{K_0}$.

Let $K_1:=2K_0$
and repeat 
the argument above for 
$K_1$ 
and the subsequence of 
$\{\pi_N\}_{N\in\N}$
constructed in the previous paragraph. 
This  yields a further subsequence of the policies 
that converges uniformly to a Lipschitz function
$\pi_\infty^1:D_{K_1}\to A$
with the Lipschitz constant bounded above by 
$C_{K_1}$.
Since the sequence we started with 
converges pointwise to 
$\pi_\infty^0$
on $D_{K_0}\subset D_{K_1}$, 
so must its every subsequence.
Hence it holds that 
$\pi_\infty^1(x)=\pi_\infty^0(x)$
for all $x\in D_{K_0}$.

For $k\in\N$, let 
$K_k:=2K_{k-1}$ 
and construct inductively 
$\pi_\infty^k:D_{K_k}\to A$
as above. 
Then the function 
$\pi_{\lim}:\R^d\to A$, 
given by $\pi_{\lim}(x):=\pi_\infty^n(x)$
for any $n\in\N$ such that $x\in D_{K_n}$,
is well-defined and locally Lipschitz. 
Let the policy $\pi_{n_k}:\R^d\to A$ be the $k$-th element of the convergent subsequence 
used to define $\pi_\infty^k:D_{K_k}\to A$.
Then, by construction, the ``diagonal'' subsequence $\{\pi_{n_k}\}_{k\in \N}$
of 
$\{\pi_N\}_{N\in\N}$
converges uniformly to 
$\pi_{\lim}$ 
on $D_K$ for any $K>0$.
\end{proof}

\begin{proof}[Proof of Theorem \ref{limitsIHM}]
Let
$\{ \pi_{n_k} \}_{k \in \N}$
be a subsequence 
of the output of~\eqref{piaIHM},
$\{ \pi_N \}_{N \in \N}$,
that converges locally uniformly to a policy
$\pi_{\lim}=\lim_{k\uparrow\infty}\pi_{n_k}$.
By~\eqref{eq:V_lim_def} and Theorem~\ref{decreasingIHM},
$V_{\pi_{n_k}}\searrow V_{\lim}$ as $k\to\infty$.
Fix $K>0$
and let 
$\tau_K:=\inf\{t\in\R_+:X^{\pi_{\lim},x}_t-x\in\partial D_K\}$ 
be the first time 
$X^{\pi_{\lim},x}$
hits the boundary of the closed ball $x+D_K$ with radius $K$,
centred at an arbitrary $x\in\R^d$.

Pick
$k \in \N$,
$t\in\R_+$
and define
$$S^k_t :=  \int_0^{t} \mathrm{e}^{-\int_0^s \alpha_{\pi_{n_k}} \left( X^{\pi_{\lim},x}_r \right) \mathrm{d}r}
f_{\pi_{n_k}} \left( X^{\pi_{\lim},x}_s \right) \mathrm{d}s +
\mathrm{e}^{-\int_0^t \alpha_{\pi_{n_k}} \left( X^{\pi_{\lim},x}_r \right) \mathrm{d}r}\,
V_{\pi_{n_k}} \left( X^{\pi_{\lim},x}_{t} \right).$$
Apply 
It\^o's formula to the process
$S^k = (S^k_t)_{t \geq 0}$
on the stochastic interval $[0,\tau_K)$
to get
\begin{align*}
& S^k_{t \wedge \tau_K} = V_{\pi_{n_k}}(x)
+ \int_0^{t \wedge \tau_K} \mathrm{e}^{-\int_0^s \alpha_{\pi_{n_k}} \left( X^{\pi_{\lim},x}_r \right) \mathrm{d}r}\,
\left( \gr V_{\pi_{n_k}} \right)^T \s_{\pi_{\lim}} \left( X^{\pi_{\lim},x}_s \right) \mathrm{d}B_s \\
& \qquad \qquad + \int_0^{t \wedge \tau_K} \mathrm{e}^{-\int_0^s \alpha_{\pi_{n_k}} \left( X^{\pi_{\lim},x}_r \right) \mathrm{d}r}
\left( f_{\pi_{n_k}} + L_{\pi_{\lim}} V_{\pi_{n_k}} - \alpha_{\pi_{n_k}} V_{\pi_{n_k}} \right)
\left( X^{\pi_{\lim},x}_s \right) \mathrm{d}s.
\end{align*}
Note that
$\s_{\pi_{\lim}}$
and
$\gr V_{\pi_{n_k}}$
are bounded on 
$D_K$
by Assumption~\ref{ass1IHM} and Proposition~\ref{operatorIHM}, respectively,
and $\alpha_{\pi_{n_k}}>\epsilon_0>0$. Hence the quadratic variation of 
the stochastic integral is bounded, making it into a true martingale. 
%Hence it is a true martingale and taking expectations yields
This fact and the equality $\alpha_{\pi_{n_k}} V_{\pi_{n_k}}-f_{\pi_{n_k}}= L_{\pi_{n_k}} V_{\pi_{n_k}}$ (Prop.~\ref{operatorIHM})
yield 
\begin{eqnarray}
\nonumber
 \E  S^k_{t \wedge \tau_K}  
 &  = & V_{\pi_{n_k}}(x) + \E  \int_0^{t \wedge \tau_K} \mathrm{e}^{-\int_0^s
\alpha_{\pi_{n_k}} \left( X^{\pi_{\lim},x}_r \right) \mathrm{d}r} \left(
f_{\pi_{n_k}} + L_{\pi_{\lim}} V_{\pi_{n_k}} - \alpha_{\pi_{n_k}} V_{\pi_{n_k}}
\right) \left( X^{\pi_{\lim},x}_s \right) \mathrm{d}s \\
 & =  & V_{\pi_{n_k}}(x) +
\E  \int_0^{t \wedge \tau_K} \mathrm{e}^{-\int_0^s \alpha_{\pi_{n_k}}
\left( X^{\pi_{\lim},x}_r \right) \mathrm{d}r} \left( L_{\pi_{\lim}} V_{\pi_{n_k}} - L_{\pi_{n_k}} V_{\pi_{n_k}} \right) \left( X^{\pi_{\lim},x}_s
\right) \mathrm{d}s.
\label{eq:pi_lim_equals_pi_k}
\end{eqnarray}

Note
$[L_{\pi_{\lim}}  - L_{\pi_{n_k}}] V_{\pi_{n_k}} =
( \mu_{\pi_{\lim}} - \mu_{\pi_{n_k}} )^T \gr V_{\pi_{n_k}}+
\frac{1}{2} \tr ( ( \s_{\pi_{\lim}} + \s_{\pi_{n_k}})^T \he V_{\pi_{n_k}} ( \s_{\pi_{\lim}} - \s_{\pi_{n_k}}))$.
Since,
for every $k$,
$V_{\pi_{n_k}}$
solves the corresponding Poisson equation in Proposition~\ref{operatorIHM}
and, by  
Assumption~\ref{ass1IHM} and Remark~\ref{rem:V_F_bounded},
the family of functions 
$\{ \s_{\pi_{n_k}},  \mu_{\pi_{n_k}},\alpha_{\pi_{n_k}}, f_{\pi_{n_k}},  V_{\pi_{n_k}} :k \in \N\}$
is uniformly bounded on the ball $x+D_K$,
Schauder's boundary estimate for
elliptic PDEs~\cite[p.~86]{FriedmanParabolic}
implies that the sequences
$\{ \gr V_{\pi_{n_k}} \}_{k \in \N}$
and
$\{ \he V_{\pi_{n_k}} \}_{k \in \N}$
are also uniformly bounded on
$x+D_K$.
Since
$\alpha_{\pi_{n_k}}>\epsilon_0>0$
for all $k\in\N$ and the limits 
$\lim_{k\uparrow\infty}\mu_{\pi_{n_k}}=\mu_{\pi_{\lim}}$
and 
$\lim_{k\uparrow\infty}\s_{\pi_{n_k}}=\s_{\pi_{\lim}}$
are uniform on $x+D_K$,
the DCT and the equality in~\eqref{eq:pi_lim_equals_pi_k} imply 
$\lim_{k\uparrow\infty} \E  S^k_{t \wedge \tau_K}  = V_{\lim}(x)$.
Hence,
the definition of $S^k$ above, 
Assumption~\ref{ass1IHM}, Remark~\ref{rem:V_F_bounded}
and a further application of the DCT
yield 
\begin{equation}
\label{eq:V_lim_final_expression}
V_{\lim}(x) =
 \E  \int_0^{t \wedge \tau_K} \mathrm{e}^{-\int_0^s \alpha_{\pi_{\lim}} \left( X^{\pi_{\lim},x}_r \right) \mathrm{d}r} f_{\pi_{\lim}} \left( X^{\pi_{\lim},x}_s \right) \mathrm{d}s
+ E_{t \wedge \tau_K},
\end{equation}
where
$E_{t \wedge \tau_K}:= \E \mathrm{e}^{-\int_0^{t \wedge \tau_K} \alpha_{\pi_{\lim}} \left( X^{\pi_{\lim},x}_r \right) \mathrm{d}r}\,
V_{\lim} \left( X^{\pi_{\lim},x}_{t \wedge \tau_K} \right)$.

By~\eqref{eq:V_lim_def} and Remark~\ref{rem:V_F_bounded}, 
the inequality
$|V_{\lim}(y)|\leq C_0$ holds for all $y\in\R^d$. 
By Assumption~\ref{ass1IHM} we hence get 
$$0\leq \limsup_{t\wedge K\to\infty}\left| E_{t\wedge K}\right| \leq C_0 \limsup_{t\wedge K\to\infty}\E \mathrm{e}^{-\epsilon_0(t \wedge \tau_K)}= 0,
$$
since $\tau_K\uparrow\infty$ as $K\uparrow\infty$.
The DCT applied 
to the first summand in~\eqref{eq:V_lim_final_expression},
as 
$t\wedge K\to\infty$,
yields
the equality 
$V_{\lim}(x) = V_{\pi_{\lim}}(x)$.
Since $x\in\R^d$ was arbitrary, the theorem follows. 
\end{proof}

\begin{proof}[Proof of Theorem \ref{verificationIHM}]
The second assertion in the theorem follows from the first one and Theorem~\ref{limitsIHM}.
We now establish the first assertion of Theorem~\ref{verificationIHM}.
Equip $A\times A$ with a product metric, e.g. 
$d_\infty((p_1,p_2),(a_1,a_2)) := \max\{d_A(a_1,p_1), d_A(a_2,p_2)\}$,
and let 
$\{ \pi_N\}_{N \in \N}$
be constructed by the~\eqref{piaIHM}. 
As in the proof of Proposition~\ref{subsequenceIHM},
$\{ (\pi_{N+1}, \pi_N) : \R^d \to A \times A \}_{N \in \N}$
are Lipschitz on a closed ball $D_K$ of radius $K>0$ with the Lipschitz constant 
$C_K$, independent of $N$. 
Hence as in the proof of  Proposition~\ref{subsequenceIHM},
there exists a subsequence 
$\{ (\pi_{1+n_k}, \pi_{n_k}) \}_{k \in \N}$
that converges uniformly on every compact subset of $\R^d$
to a locally Lipschitz function 
$(\tilde \pi_{\lim}, \pi_{\lim}) : \R^d \to A \times A$.

Pick any $x\in\R^d$, a policy $\Pi\in\A(x)$, $K>0$
and let 
$\tau_K:=\inf\{t\in\R_+:X^{\Pi,x}_t-x\in\partial D_K\}$ 
be the first time the controlled process 
$X^{\Pi,x}$
hits the boundary of the closed ball $x+D_K$ with radius $K$
(centred at  $x$). Since $\Pi_s\in A$ for all $s\in\R_+$, the~\eqref{piaIHM} implies the inequality
\begin{equation}
\label{eq:Scaled_Pia_estimate}
S_{\Pi_s} (f_{\Pi_s} + L_{\Pi_s} V_{\pi_{n_k}} - \alpha_{\Pi_{s}} V_{\pi_{n_k}}) \geq 
 S_{\pi_{n_k + 1}} (f_{\pi_{n_k + 1}} + L_{\pi_{n_k + 1}} V_{\pi_{n_k}} - \alpha_{\pi_{n_k+1}} V_{\pi_{n_k}}) \quad \text{on $\R^d$.}
\end{equation}
Denote
$\mathcal{L}_\pi h := L_\pi h % \frac{1}{2} \tr \left( \s_\pi^T \he h \s_\pi \right) + \mu_\pi^T \gr h
- \alpha_\pi h + f_\pi$ for any policy $\pi$ and $h \in \C^2(\R^d)$.
Then, for 
$k \in \N$,
we find that
\begin{align}
\nonumber
& \E \left( \int_0^{t \wedge \tau_K} \mathrm{e}^{-\int_0^s \alpha_{\Pi_r} \left( X^{\Pi,x}_r \right) \mathrm{d}r} f_{\Pi_s} \left( X^{\Pi,x}_s \right) \mathrm{d}s +
\mathrm{e}^{-\int_0^{t \wedge \tau_K} \alpha_{\Pi_r} \left( X^{\Pi,x}_r \right) \mathrm{d}r}\,
V_{\pi_{n_k}} \left( X^{\Pi,x}_{t \wedge \tau_K} \right) \right) \\
\nonumber
& \stackrel{\text{It\^{o}}}{=} V_{\pi_{n_k}}(x) + \E  \int_0^{t \wedge \tau_K}
\mathrm{e}^{-\int_0^s \alpha_{\Pi_{r}} \left( X^{\Pi,x}_r \right) \mathrm{d}r}
\left( f_{\Pi_s} + L_{\Pi_s} V_{\pi_{n_k}} - \alpha_{\Pi_{s}} V_{\pi_{n_k}}
\right) \left( X^{\Pi,x}_s \right) \mathrm{d}s  \\
& \geq V_{\pi_{n_k}}(x) +\frac{\epsilon_S}{M_S} \E  \int_0^{t \wedge \tau_K} \mathrm{e}^{-\int_0^s \alpha_{\Pi_{r}} \left( X^{\Pi,x}_r \right) \mathrm{d}r} 
\left(  \mathcal{L}_{\pi_{n_k + 1}} V_{\pi_{n_k}} \right) \left( X^{\Pi,x}_s \right) \mathrm{d}s,  
\label{eq:final_verification}
\end{align}
where the last inequality follows from 
Assumption~\ref{ass2_and_a_half_IHM} and inequality~\eqref{eq:Scaled_Pia_estimate}.

The next task is to take the limit as $k\to\infty$
on both sides of inequality~\eqref{eq:final_verification}. 
%Note first that the policies
%$\tilde\pi_{\lim}$ and $\pi_{\lim}$ need not be the same. 
%However, 
Since the sequence 
$\{ \pi_{1+n_k}\}_{k \in \N}$
%(resp.
%$\{ \pi_{n_k} \}_{k \in \N}$)
converges locally uniformly to the locally Lipschitz policy
$\tilde\pi_{\lim}$ 
(resp. $\pi_{\lim}$),
Theorem~\ref{limitsIHM} implies
$V_{\tilde\pi_{\lim}} = V_{\lim}$
(resp. $V_{\pi_{\lim}}= V_{\lim}$).
Proposition~\ref{operatorIHM}
implies
$\mathcal{L}_{\tilde\pi_{\lim}} V_{\lim} =0=\mathcal{L}_{\pi_{\lim}} V_{\lim}$.
Hence we can express
$\mathcal{L}_{\pi_{n_k + 1}} V_{\pi_{n_k}}=  \mathcal{L}_{\pi_{n_k + 1}}V_{\pi_{n_k}}-\mathcal{L}_{\tilde\pi_{\lim}}V_{\pi_{n_k}}+
\mathcal{L}_{\tilde\pi_{\lim}} V_{\pi_{n_k}}-\mathcal{L}_{\tilde\pi_{\lim}} V_{\lim}$.
By Schauder's boundary estimate for
elliptic PDEs~\cite[p.~86]{FriedmanParabolic},
the sequences 
$\{ \gr V_{\pi_{n_k}} \}_{k \in \N}$
and
$\{ \he V_{\pi_{n_k}} \}_{k \in \N}$
are uniformly bounded on
$x+D_K$.
By Assumption~\ref{ass1IHM} and Remark~\ref{rem:V_F_bounded},
the bounded sequence 
$\{(\s_{\pi_{n_k+1}},  \mu_{\pi_{n_k+1}},\alpha_{\pi_{n_k+1}}, f_{\pi_{n_k+1}},  V_{\pi_{n_k+1}}) \}_{k\in\N}$
tends to the limit
$(\s_{\tilde\pi_{\lim}},  \mu_{\tilde\pi_{\lim}},\alpha_{\tilde\pi_{\lim}}, f_{\tilde\pi_{\lim}},  V_{\tilde \pi_{\lim}})$
uniformly on 
$x+D_K$
as $k\uparrow\infty$.
Hence,
so does 
\begin{equation}
\label{eq:Gen_Conv}
 \mathcal{L}_{\pi_{n_k + 1}}V_{\pi_{n_k}}-\mathcal{L}_{\tilde\pi_{\lim}}V_{\pi_{n_k}}=
 [ L_{\pi_{n_k+1}}-L_{\tilde\pi_{\lim}}] V_{\pi_{n_k}} -  (\alpha_{\pi_{n_k+1}}- \alpha_{\tilde \pi_{\lim}})V_{\pi_{n_k}} +
(f_{\pi_{n_k+1}}- f_{\tilde \pi_{\lim}})\to0. %\to 0\quad\text{as $k\to\infty$,}
\end{equation}

By the elliptic version of Theorem 15
%\footnote{
%The theorem is stated for parabolic equations, but is also valid for the
%ellpitic ones; in the same way as the parabolic version follows from Theorem 5,
%p.\ 64, the elliptic version follows from Interior Estimates, p.\ 86. Given the
%setting from the Interior Estimates, the theorem establishes the convergence of
%solutions of uniformly elliptic differential equations provided that the
%coefficients converge.} 
in~\cite[p.\ 80]{FriedmanParabolic} applied to the 
family of PDEs 
$\mathcal{L}_{\pi_{n_k}}V_{\pi_{n_k}}=0$, $k\in\N$,
there exists a subsequence of 
$\{V_{\pi_{n_k}}\}_{k\in\N}$
(again denoted by 
$\{V_{\pi_{n_k}}\}_{k\in\N}$),
such that the corresponding sequence 
$\{( V_{\pi_{n_k}}, \gr V_{\pi_{n_k}}, \he V_{\pi_{n_k}})\}_{k\in\N}$
converges uniformly on the closed ball
$x+D_K$
to
$(V_{\pi_{\lim}}, \gr V_{\pi_{\lim}}, \he V_{\pi_{\lim}})=(V_{\lim}, \gr V_{\lim}, \he V_{\lim})$.
Hence it follows that 
\begin{equation}
\label{eq:Gen_Conv_1}
\mathcal{L}_{\tilde\pi_{\lim}} V_{\pi_{n_k}}-\mathcal{L}_{\tilde\pi_{\lim}} V_{\lim}= 
L_{\tilde\pi_{\lim}} (V_{\pi_{n_k}} - V_{\lim})-\alpha_{\tilde\pi_{\lim}}(V_{\pi_{n_k}} - V_{\lim})\to0,
\qquad\text{as $k\to\infty$.}
\end{equation}
Equations~\eqref{eq:Gen_Conv} and~\eqref{eq:Gen_Conv_1} imply that 
$\mathcal{L}_{\pi_{n_k + 1}} V_{\pi_{n_k}}\to0$
as $k\to\infty$
uniformly on the ball 
$x+D_K$.

Apply the DCT to the right-hand side of~\eqref{eq:final_verification} 
and Assumption~\ref{ass1IHM} and  Remark~\ref{rem:V_F_bounded} to its left-hand side:
\begin{align*}
V_{\lim}(x) \leq 
 \E  \int_0^{t \wedge \tau_K} \mathrm{e}^{-\int_0^s \alpha_{\Pi_r} \left( X^{\Pi,x}_r \right) \mathrm{d}r} f_{\Pi_s} \left( X^{\Pi,x}_s \right) \mathrm{d}s +
C_0 \E \mathrm{e}^{-\epsilon_0 (t \wedge \tau_K) }.
\end{align*}
Since this inequality holds for all $K,t>0$
and 
$\tau_K\uparrow\infty$ as $K\uparrow\infty$, the inequality $V_{\lim}(x)\leq V_\Pi(x)$ follows by the 
DCT as $t\wedge K\to\infty$ (cf.
the last paragraph of the proof of Theorem~\ref{limitsIHM}).
\end{proof}

%%%%%%%%%%%%%%%%%%%%%%%%%%%%%%%%%%%%%%%%%%%%%%%%%%%%%%%%%%

\subsection{Auxiliary results - the one-dimensional case}
\label{subsec:aux_one_dim}
%%%%%%%%%%%%%%%%%%%%%%%%%%%%%%%%%%%%%%%%%%%%%%%%%%%%%%%%%%
Throughout Sections~\ref{subsec:aux_one_dim} and~\ref{subsec:Proofs_OneDim},
define
$\tau_c^d(Z):= \inf \{ t \geq 0;\;Z_t \in\{c,d\} \}$ ($\inf \emptyset = \infty$)
for any continuous stochastic process $(Z_t)_{t\in\R_+}$ in $\R$
and $-\infty\leq c<d\leq\infty$.

\begin{lemma}
\label{boundaryIHP}
For any Markov policy
$\pi:(a,b)\to A$,
the payoff function
$V_\pi : (a,b) \to \R$
can be continuously extended
by defining
$V_\pi(a) := g(a)$
if
$a > -\infty$
and
$V_\pi(b) := g(b)$
if
$b < \infty$.
\end{lemma}

\begin{proof}
Let
$\{ x_n \}_{n \in \N}$
be a decreasing sequence in
$(a,b)$
that converges to
$a > -\infty$.
We now prove that $\lim_{n \to \infty} V_\pi(x_n) = g(a)$
(the argument for
$b$
is analogous).

Pick arbitrary
$\epsilon > 0$.
Since
$\mu$
is bounded and
$\s^2$
bounded and bounded away from
$0$,
a simple coupling argument 
yields that 
the process
$X^{\pi,x_n}$
can be bounded 
by a Brownian motion with drift
so that
$\p \left( \tau_a^{\infty} \left( X^{\pi,x_n} \right) > \epsilon \right) < \epsilon$
and
$\p \left( \tau_a^{\infty} \left( X^{\pi,x_n} \right) > \tau_{-\infty}^{b} \left( X^{\pi,x_n} \right) \right) < \epsilon$
hold for large 
$n \in \N$.
Hence, there exists $n_0\in\N$ such that 
\begin{align*}
\p \left( \tau_a^{\infty} \left( X^{\pi,x_n} \right) > \epsilon \wedge \tau_{-\infty}^{b} \left( X^{\pi,x_n} \right) \right) 
\leq \p \left( \tau_a^{\infty} \left( X^{\pi,x_n} \right) > \epsilon \right)
+ \p \left( \tau_a^{\infty} \left( X^{\pi,x_n} \right) > \tau_{-\infty}^{b} \left( X^{\pi,x_n} \right) \right) 
< 2 \epsilon
\end{align*}
for all $n\geq n_0$.
Define the quantities
$B_a^b:= |  \mathrm{e}^{-\int_0^{\tau_a^b ( X^{\pi,x_n} )} \alpha_{\pi} ( X^{\pi,x_n}_s ) \mathrm{d}t}
g ( X^{\pi,x_n}_{\tau_a^b(X^{\pi,x_n})} ) - g(a) |$
and
$A_a^b:=
\int_0^{\tau_a^b ( X^{\pi,x_n} )} \mathrm{e}^{-\int_0^t \alpha_{\pi} ( X^{\pi,x_n}_s )
\mathrm{d}s} |  f_{\pi} ( X^{\pi,x_n}_t ) | \mathrm{d}t$
and the event 
$C:=\{ \tau_a^{\infty} \left( X^{\pi,x_n} \right) \leq \epsilon \wedge \tau_{-\infty}^{b} \left( X^{\pi,x_n} \right) \}$.
Then we have
\begin{align*}
|V_\pi(x_n) - g(a)| \leq
 \E\left(  (A_a^b+B_a^b) \I_{\Omega\setminus C}+
(A_a^b+B_a^b) \I_{C} \right). 
\end{align*}

We now show that there exists
$M > 0$,
which does not depend on $\epsilon$,
such that 
$|V_\pi(x_n) - g(a)|$
is bounded above by
$4M \epsilon$
for all 
$n \geq n_0$.
The expectation on the event 
$\Omega\setminus C$, which
has probability less than
$2 \epsilon$,
is smaller than 
$2M \epsilon$
since
$f,g$
are bounded
and
$\alpha\geq\epsilon_0>0$.
On the event $C$ 
we have 
$\tau_a^b ( X^{\pi,x_n} )\leq \epsilon$,
which 
implies
$\E A_a^b \I_C<M\epsilon$.
On
$C$
it holds that 
$X^{\pi,x_n}_{\tau_a^b(X^{\pi,x_n})}=a$.
Hence, 
the elementary inequality
$1 - \mathrm{e}^{-x} \leq x$ for
$x \geq 0$,
yields an upper bound on 
$\E B_a^b \I_C$
of the form 
$ | g(a)| \E \left( \int_0^{\epsilon} \alpha_{\pi} \left( X^{\pi,x_n}_t \right) \mathrm{d}t \right)$.
This concludes the proof.
\end{proof}

Lemma~\ref{martingaleIHO} is the analogue of Lemma~\ref{martingaleIHM}
with an analogous proof, which we omit for brevity. 

\begin{lemma}
\label{martingaleIHO}
The following holds for every Markov policy
$\pi$,
$x \in (a,b)$
and stopping time
$\rho$:
\begin{align*}
& \E \Bigg( F_{\tau_a^b( X^{\pi,x} )}(X^{\pi,x}) 
 + \mathrm{e}^{-\int_0^{\tau_a^b(X^{\pi,x})} \alpha_\pi \left( X^{\pi,x}_t \right) \mathrm{d}t}\, g \left( X^{\pi,x}_{\tau_a^b(X^{\pi,x})} \right)
\I_{\{ \tau_a^b(X^{\pi,x}) < \infty \}} \Bigg\vert \F_\rho \Bigg) = M_\rho, 
\end{align*}
where
$M_r := 
F_{r \wedge \tau_a^b( X^{\pi,x} )}(X^{\pi,x}) 
+ \I_{\{ r < \infty\}} \mathrm{e}^{-\int_0^{r \wedge \tau_a^b ( X^{\pi,x} )} \alpha_\pi
( X^{\pi,x}_s ) \mathrm{d}s}\,
V_\pi( X^{\pi,x}_{r \wedge \tau_a^b ( X^{\pi,x})})$,
for $r\in[0,\infty]$.
In particular, the process
$M=(M_r)_{r\in[0,\infty]}$
is a uniformly integrable martingale. 
\end{lemma}

\subsection{Proofs of results in Section~\ref{ch3IHO}}
\label{subsec:Proofs_OneDim}

\begin{proof}[Proof of Proposition \ref{operatorIHM} in dimension one]
Recall that Assumption~\ref{ass1IHM} holds. We need to show that 
for any locally Lipschitz Markov policy
$\pi:(a,b)\to A$
we have
$V_\pi \in \C^2((a,b))$
and
$L_\pi V_\pi - \alpha_\pi V_\pi + f_\pi = 0$.

Let
$a < a' < a'' < x < b'' < b' < b$,
and for any
$c < d$
denote
$\tau_{c}^{d} := \tau_{c}^{d}(X^{\pi,x})$.
Let
$v \in \C^2((a',b')) \cap \C([a',b'])$
be the unique solution of the boundary value problem
$L_\pi v - \alpha_\pi v + f_\pi = 0$,
$v(a') = V_\pi(a')$, $v(b') = V_\pi(b')$,
guaranteed to exist by Theorem 19 in \cite[p.\ 87]{FriedmanParabolic}, which is applicable by Assumption~\ref{ass1IHM}. 
Let $ S_t^{a'',b''}  :=    F_{t \wedge \tau_{a''}^{b''}}(X^{\pi,x}) +
\mathrm{e}^{-\int_0^{t \wedge \tau_{a''}^{b''}} \alpha_{\pi} ( X^{\pi,x}_r ) \mathrm{d}r}
v ( X^{\pi,x}_{t \wedge \tau_{a''}^{b''}})$.
Then,  
by It\^o's formula on
$[0,\tau_{a''}^{b''}]$
and the definition of $v$,
the process
$S^{a'',b''} = (S_t^{a'',b''})_{t \geq 0}$
satisfies
\begin{eqnarray*}
S_t^{a'',b''} & =  &  
 v(x) + \int_0^{t \wedge \tau_{a''}^{b''}} \mathrm{e}^{-\int_0^s \alpha_{\pi} \left( X^{\pi,x}_r \right) \mathrm{d}r}
\s_\pi v' \left( X^{\pi,x}_s \right) \mathrm{d}B_s.
\end{eqnarray*}
Hence 
$S^{a'',b''}$
is clearly a uniformly integrable martingale 
and the following equalities hold:
$\lim_{t\uparrow\infty}\E| S_t^{a'',b''}- S_\infty^{a'',b''}|=0$  
and
$v(x)=\E  S_\infty^{a'',b''}$.
Define $S^{a',b'}$ by
substituting 
$\tau_{a''}^{b''}$
in the definition of 
$S^{a'',b''}$
with  $\tau_{a'}^{b'}$.
Since $X^{\pi,x}$ is continuous,
we have
$\lim_{a''\downarrow a',b''\uparrow b'}\tau_{a''}^{b''}=\tau_{a'}^{b'}$ a.s.
Hence, by the DCT, 
$v(x) = \lim_{a''\downarrow a',b''\uparrow b'} \E  S_\infty^{a'',b''}  = \E  S_\infty^{a',b'}$.

Note that 
the boundary conditions for
$v$,
the fact
$X^{\pi,x}_{\tau_{a'}^{b'}}\in\{a',b'\}$ 
and
Lemma~\ref{martingaleIHO} (with
$\rho=\tau_{a'}^{b'}$)
imply 
$ S_\infty^{a',b'} 
 = \E ( F_{\tau_{a}^{b}}(X^{\pi,x})
+ \mathrm{e}^{-\int_0^{\tau_{a}^{b}}
\alpha_{\pi} ( X^{\pi,x}_r ) \mathrm{d}r}
g ( X^{\pi,x}_{\tau_{a}^{b}} ) \I_{\{ \tau_{a}^{b} < \infty \}} \vert \F_{\tau_{a'}^{b'}} )$.
Taking expectations on both sides of this equality yields
$v(x)  
= V_\pi(x)$.
\end{proof}

\begin{proof}[Proof of Theorem \ref{decreasingIHM} in dimension one]
We claim that under Assumptions~\ref{ass1IHM}--\ref{ass2_and_a_half_IHM}, 
the inequality  $V_{\pi_{n+1}}(x) \leq V_{\pi_{n}}(x)$ holds
for all $x\in(a,b)$ and $n\in\N$,
where $\pi_{n+1}$ is defined in~\eqref{piaIHO}.

Define the process
$Y$
as in~\eqref{eq:S_defin} in Section~\ref{subsec:Proofs_Multidim}
and consider the stopped process
$Y_{\cdot\wedge \tau_{a'}^{b'}}$, where
$a < a' < x < b' < b$
and
$\tau_{c}^{d} := \tau_{c}^{d}(X^{\pi_{n+1},x})$
for any
$c < d$.
Then the proof follows the same steps as the proof of 
Theorem \ref{decreasingIHM} in 
Section~\ref{subsec:Proofs_Multidim}.
The only difference is that in the penultimate line of the proof 
of Section~\ref{subsec:Proofs_Multidim}
we apply
the DCT and Lemma~\ref{boundaryIHP} (instead of the DCT only)
to obtain
$V_{\pi_n}(x)\geq V_{\pi_{n+1}}(x)$.
\end{proof}

The proof of the one-dimensional case of Proposition~\ref{subsequenceIHM}  
is completely analogous to the multi-dimensional one and is hence omitted.

\begin{proof}[Proof of Theorem \ref{limitsIHM} in one dimension]
We need to show that 
$V_{\lim}(x) = V_{\pi_{\lim}}(x)$
holds for 
all $x\in(a,b)$.
The proof follows along the same lines as in the multi-dimensional case of Section~\ref{subsec:Proofs_Multidim}.
The only difference lies in the fact that we stop the process 
$X^{\pi_{\lim},x}$ 
at 
$\tau_{a'}^{b'}(X^{\pi_{\lim},x})$, 
where
$a<a'<x<b'<b$,
and take the limit as
$(a',b',t)\to(a,b,\infty)$.
\end{proof}

The verification lemma in the one-dimensional case is established 
exactly as in the proof of Theorem \ref{verificationIHM}. The details are omitted.


\begin{thebibliography}{10}

\bibitem{Ari} A.\ Arapostathis, On the Policy Iteration Algorithm for Nondegenerate Controlled Diffusions Under the Ergodic Criterion. \emph{ Optimization, Control, and Applications of Stochastic Systems}, 1--12, 2012.

\bibitem{Ari_Borkar} A.\ Arapostathis and V.S.\ Borkar, Uniform recurrence properties of controlled diffusions and applications to optimal control. 
\emph{SIAM J. Control Optim.}, 48(7), 152--160, 2010.

\bibitem{Erhan} P. Azimzadeh, E. Bayraktar and G. Labahn, Convergence of approximation schemes for weakly nonlocal second order equations.
 \emph{arXiv:1705.02922v2}, 2017.

\bibitem{Bertsekas} P.\ Bertsekas, Approximate Policy Iteration: a Survey and Some New Methods. \emph{Journal of Control Theory and Applications}, 9(3):310--335, 2011.

\bibitem{Borodin} A.\ N.\ Borodin, P.\ Salminen, \emph{Handbook of Brownian
Motion -- Facts and Formulae}.  Birkhauser, Basel, second edition, 2002.


\bibitem{Cerny_Engelbert} A.\ S.\ Cherny and H.-J. Engelbert, \emph{Singular Stochastic Differential Equations}.
Lecture Notes in Mathematics,
Springer-Verlag, Berlin Heidelberg, 2005.


\bibitem{Doshi} B.\ T.\ Doshi, Continuous Time Control of Markov Processes on an Arbitrary State Space: Discounted Rewards. \emph{The Annals of Statistics}, 4(6):1219--1235, 1976.

\bibitem{Dynkin} E.B.\ Dynkin, \emph{Markov Processes, vol~I}.  Springer-Verlag, Berlin, 1965.


\bibitem{Fournie_et_al} 
E.\ Fourni\'e, J.-M.\ Lasry, J.\ Lebuchoux,
P.-L.\ Lions, N.\ Touzi, Applications of Malliavin calculus to Monte Carlo methods in finance.  \emph{Finance and Stochastics}, 3:391--412, 1999.

\bibitem{FriedmanParabolic} A.\ Friedman, \emph{Partial Differential Equations of Parabolic Type}.
Prentice-Hall, Englewood Cliffs, N.J., 1964.


\bibitem{Lerma} O.\ Hernández-Lerma and J.\ B.\ Lasserre, Policy iteration for average cost Markov control processes on Borel spaces. \emph{Acta Applicandae Mathematicae}, 47(2):125--154, 1997.

\bibitem{Hordi} A.\ Hordijk and M.\ L.\ Puterman, On the convergence of policy iteration in finite state undiscounted Markov decision processes: the unichain case. \emph{Mathematics of Operations Research}, 12(1):163--176, 1987.

\bibitem{Howard} R.\ A.\ Howard, \emph{Dynamic Programming and Markov Processes}.
 MIT Press, Cambridge, 1960.

%\bibitem{our} S.\ D.\ Jacka, A.\ Mijatovi\'c and D.\ \v{S}iraj,
%Mirror and Synchronous Couplings of Geometric Brownian Motions. \emph{Stochastic Processes and their Applications}, %123(2):1055--1069, 2014.

\bibitem{Karatzas} I.\ Karatzas and S.\ E.\ Shreve, \emph{Brownian Motion and
Stochastic Calculus}. Springer-Verlag, New York, second edition, 1991.

\bibitem{Krylov} N.\ V.\ Krylov, \emph{Controlled Diffusion Processes}.
Springer-Verlag, New York, 1980.

%\bibitem{Lady} O.\ A.\ Ladyzhenskaya, N.\ N.\ Ural'tseva, \emph{Linear and Quasilinear Elliptic Equations}.
%Academic Press, New York, 1968.

\bibitem{Parr} M.\ G.\ Lagoudakis and R.\ Parr, Least-Squares Policy Iteration. \emph{Journal of Machine Learning Research}, 4:1107-1149, 2003.

\bibitem{Lasser} J.\ B.\ Lasserre, A new policy iteration scheme for Markov decision processes using Schweitzer's formula.
\emph{Journal of Applied Probability}, 31(1):268--273, 1994.

\bibitem{article6} T.\ Lindvall and L.\ C.\ G.\ Rogers, Coupling of
Multidimensional Diffusions by Reflection.  \emph{The Annals of Probability},
14(3):860--872, 1986.

\bibitem{Mahadevan} S.\ Mahadevan, Representation Policy Iteration. ArXiv:1207.1408 [cs.AI], 2012.

\bibitem{Meyn} S.\ P.\ Meyn, The policy iteration algorithm for average reward Markov decision processes with general state space. \emph{IEEE Transactions on Automatic Control}, 42(12):1663--1680, 1997.

\bibitem{Protter} P.\ E.\ Protter, \emph{Stochastic Integration and Differential Equations}.
Stochastic Modelling and Applied Probability, Springer-Verlag, Berlin Heidelberg, 2004.

\bibitem{Rust} J.\ P.\ Rust, A Comparison of Policy Iteration Methods for Solving Continuous-State, Infinite-Horizon Markovian Decision Problems Using Random, Quasi-Random, and Deterministic Discretizations. Available at http://ssrn.com/abstract=37768, 1997.

\bibitem{Santos} M.\ S.\ Santos and J.\ Rust, Convergence properties of policy iteration.
\emph{SIAM Journal on Control and Optimization}, 42(6):2094--2115, 2004.

%\bibitem{Stroock} D.\ W.\ Stroock and S.\ R.\ S.\ Varadhan, \emph{Multidimensional Diffusion Processes}. Springer-Verlag, Berlin, 2006.

%\bibitem{Lewis} K.\ G.\ Vamvoudakis and F.\ L.\ Lewis, Online Actor–Critic
%Algorithm to Solve the Continuous-Time Infinite Horizon Optimal Control
%Problem. \emph{Automatica}, 46(5):878--888, 2010.

\bibitem{Zhu} Q.\ X.\ Zhu, X.\ S.\ Yang and C.\ X.\ Huang,
Policy iteration for continuous-time average reward Markov decision processes in Polish spaces.
\emph{Abstract and Applied Analysis}, 2009, Article ID103723, 17 pages.


\end{thebibliography}
\end{document}